\documentclass[]{amsart}

\usepackage{mathtools}
\usepackage{amssymb, graphicx, scalerel}
\usepackage{wasysym}
\usepackage{stackengine}
\usepackage{adjustbox}
\usepackage[all,2cell,cmtip]{xy}
\usepackage{tikz}
\usepackage{hyperref}
\usepackage{MnSymbol}
\usepackage{xfrac}

\makeatletter
\newcommand{\equaldot}{\mathrel{\mathpalette\eqd@t\relax}}
\newcommand{\eqd@t}[2]{%
	\ooalign{%
		$\m@th#1=$\cr
		\hidewidth$\m@th#1\cdot$\hidewidth\cr
	}%
}
\makeatother

\newtheorem{theorem}{Theorem}[section]
\newtheorem{proposition}[theorem]{Proposition}
\newtheorem{lemma}[theorem]{Lemma}
\newtheorem{corollary}[theorem]{Corollary}

\theoremstyle{definition}
\newtheorem{definition}[theorem]{Definition}

\newtheorem{example}[theorem]{Example}
\newtheorem{remark}[theorem]{Remark}

\newcommand{\colim}{\operatornamewithlimits{colim}}
\newcommand{\hocolim}{\operatornamewithlimits{hocolim}}

\newcommand{\C}{\mathcal{C}}
\newcommand{\rotboxtimes}{\scriptsize\rotatebox[origin=c]{45}{$\boxtimes$}}
\newcommand{\rotbox}{\scriptsize\rotatebox[origin=c]{45}{$\square$}}
\newcommand{\specialdots}{\mathbin{\tikz{
			\draw (0,0) circle (0.65pt);
			\fill (0,-.067) circle (0.45pt);
			\fill (0,.067) circle (0.45pt);
			\fill (-.067,0) circle (0.45pt);
			\fill (.067,0) circle (0.45pt);
		}}}
\newcommand{\plusdots}{\mathbin{\tikz{
			\draw (-1pt,0) -- (1pt,0);
			\draw (0,-1pt) -- (0,1pt);
			\fill (0,-.067) circle (0.45pt);
			\fill (0,.067) circle (0.45pt);
			\fill (-.067,0) circle (0.45pt);
			\fill (.067,0) circle (0.45pt);
}}}
\newcommand{\nsfrac}[2]{\text{\raisebox{0.25ex}{$#1$}}\kern-0.1em/\kern-0.15em\text{\raisebox{-0.25ex}{$#2$}}}
\newcommand{\sgn}{\mathrm{sgn}}

\title{Polynomial $2$-monads and delooping}

\author[Florian De Leger]{Florian De Leger}

\address{School of Applied Mathematical and Physical Sciences, National Technical University of Athens}

\email{fdeleger@mail.ntua.gr}

\begin{document}
	
\begin{abstract}
	Using the homotopy theory of polynomial monads developed by Batanin and Berger and extended to the $2$-categorical context by Weber, we prove the cofinality of a particular morphism of polynomial $2$-monads. We apply our result to give a new proof of the delooping of derived mapping spaces of infinitesimal bimodules due to Ducoulombier and Turchin.
\end{abstract}
	
\maketitle

\tableofcontents

\section*{Introduction}

There is a remarkable connection between embedding spaces and mapping spaces of operads. This connection is exhibited by the following weak equivalence of spaces, which holds for $n \geq m+3$:
\begin{equation}\label{equationdeloopingembeddingspace}
	\Omega^{m+1} \mathrm{SOp}^\mathsf{h} (\mathcal{C}_m,\mathcal{C}_n) \xrightarrow{\sim} \overline{\mathrm{Emb}}(\mathbb{R}^m,\mathbb{R}^n).
\end{equation}
In the formula above, $\Omega^{m+1}$ is the $m+1$-fold {loop space functor}, $\mathrm{SOp}^\mathsf{h}(-,-)$ is the {derived mapping space} in the category of symmetric topological operads and $\mathcal{C}_m$ is the little $m$-cubes operad. The space on the right hand side is the homotopy fibre of the canonical inclusion
\[
\mathrm{Emb}(\mathbb{R}^m,\mathbb{R}^n) \to \mathrm{Imm}(\mathbb{R}^m,\mathbb{R}^n)
\]
where $\mathrm{Emb}(-,-)$ (resp. $\mathrm{Imm}(-,-)$) is the space of embeddings (resp. immersions) which agree with the standard inclusion outside a compact. This $m+1$-fold delooping was proved by Boavida de Brito and Weiss \cite{boavida} using configuration categories. It can also be deduced from the results of \cite{aroneturchin,ducoulombier,ducoulombierturchin}. Indeed, Arone and Turchin proved that there is a weak equivalence
\[
\mathrm{IBimod}_{\mathcal{C}_m}^\mathsf{h}(\mathcal{C}_m,\mathcal{C}_n) \xrightarrow{\sim} \overline{\mathrm{Emb}}(\mathbb{R}^m,\mathbb{R}^n)
\]
where $\mathrm{IBimod}_{\mathcal{C}_m}^\mathsf{h}(-,-)$ is the derived mapping space in the category of infinitesimal $\mathcal{C}_m$-bimodules. This weak equivalence is the higher dimensional analogue of Sinha's result \cite{sinha1} concerning the space of long knots. Combining this weak equivalence with the delooping \cite{ducoulombier}
\begin{equation}\label{equationducoulombier}
\Omega \mathrm{SOp}^\mathsf{h}(\C_m,\C_n) \xrightarrow{\sim} \mathrm{Bimod}_{\C_m}^\mathsf{h} (\C_m,\C_n),
\end{equation}
where $\mathrm{Bimod}_{\mathcal{C}_m}^\mathsf{h}(-,-)$ is the derived mapping space in the category of $\mathcal{C}_m$-bimodules, and the $m$-fold delooping \cite{ducoulombierturchin}
\begin{equation}\label{equationducoulombierturchin}
\Omega^m \mathrm{Bimod}_{\C_m}^\mathsf{h} (\C_m,\C_n) \xrightarrow{\sim} \mathrm{IBimod}_{\C_m}^\mathsf{h} (\C_m,\C_n),
\end{equation}
one does indeed recover \eqref{equationdeloopingembeddingspace}.

%Both \eqref{equationtdh} and \eqref{equationsinha} admit the following higher dimensional analogues. Arone and Turchin \cite{aroneturchin} proved that there is a weak equivalence
%\[
%	\mathrm{IBimod}_{\mathcal{C}_m}(\mathcal{C}_m,\mathcal{C}_n) \to \overline{\mathrm{Emb}}(\mathbb{R}^m,\mathbb{R}^n)
%\]

The deloopings \eqref{equationducoulombier} and \eqref{equationducoulombierturchin} are the higher dimensional version of the deloopings due to Dwyer-Hess \cite{dwyerhess} and Turchin \cite{turchin}. In \cite{batanindeleger}, Batanin and the author gave a new proof of the deloopings of Dwyer-Hess and Turchin, using the homotopy theory of polynomial monads. %These deloopings are the non-symmetric analogues, for $m=1$, of the deloopings \eqref{equationducoulombier} and \eqref{equationducoulombierturchin}. % extended some classical results of homotopy theory from small categories to polynomial monads. As an application, we , which are the non-symmetric analogues, for $m=1$, of the deloopings \eqref{equationducoulombier} and \eqref{equationducoulombierturchin}. 
Unfortunately, the framework of polynomial monads does not allow us to tackle the higher dimensional delooping problems. The issue is that structures related to the little $m$-cubes operad can not be encoded by polynomial monads. The solution to overcome this issue is to work with polynomial $2$-monads, which is what we will do in this paper.

The central notion in the homotopy theory of polynomial monad is Batanin's \emph{classifier of internal $S$-algebras inside categorical $T$-algebras}, associated to a cartesian morphism between cartesian monads $f: S \to T$ \cite{batanin}. A prototypical example of an internal algebra classifier is the category $\Delta_+$ of finite ordinals and order-preserving maps, associated to $id: \mathbf{M} \to \mathbf{M}$, where $\mathbf{M}$ is free monoid monad. The addition of ordinals gives a strict monoidal structure on $\Delta_+$. In fact, it is the free strict monoidal category containing a monoid \cite[Section VII.2]{maclane}. This means that it has the following universal property: for any strict monoidal category $A$, monoids in $A$ are equivalent to strict monoidal functors from $\Delta_+$ to $A$. It was proved by Batanin \cite[Theorem 7.3]{batanin} that classifiers can be computed as a bar-construction. For example, the truncated nerve of the category $\Delta_+$ is:
\[
\xymatrix@C=1.5pc{
	\mathbf{M}^3 1 \ar@<1.5ex>[rr]^{\mu_{\mathbf{M}1}} \ar[rr]|{\mathbf{M}\mu_1} \ar@<-1.5ex>[rr]_{\mathbf{M}^2 !} && \mathbf{M}^2 1 \ar@<1.5ex>[rr]^-{\mu_1} \ar@<-1.5ex>[rr]_-{\mathbf{M}!} && \mathbf{M}1, \ar[ll]|-{\mathbf{M} \eta_1}
}
\]
where $\eta$ and $\mu$ are the unit and multiplication of the monad and $1$ is the singleton set.
This formula has been extended by Weber to compute the classifier associated to any cartesian morphism between cartesian $2$-monads. Indeed, Weber \cite[Theorem 5.4.3]{weber} has proved that such classifiers can be computed as the category of corners (see Definition \ref{definitioncategoryofcorners}) of a crossed double category (see Definition \ref{definitioncrosseddoublecategory}).

Batanin's classifiers allowed us to introduce a notion of a \emph{homotopically cofinal} morphism of polynomial monads \cite[Definition 5.6]{batanindeleger}, extending the notion of a \emph{homotopy left cofinal} functor between small categories \cite[Definition 19.6.1]{hirschhorn}. Just like (homotopy) cofinal functors preserve homotopy (limits), homotopically cofinal morphisms of polynomial monads \emph{preserve} derived mapping spaces. One of the main results of \cite{batanindeleger} is the cofinality of two specific morphisms of polynomial monads, from which we recovered the deloopings of Dwyer-Hess and Turchin. Similarly, in this paper, we will construct a morphism of polynomial $2$-monads
\begin{equation}\label{equationhomotopycofinal}
	f:\mathcal{K}_m^\fivedots \to \mathcal{K}_m^{\rotboxtimes}.
\end{equation}
Applying Weber's formula for the computation of internal algebra classifiers, we will prove (see Theorem \ref{theoremcofinality}) that this morphism is homotopically cofinal.

The morphism \eqref{equationhomotopycofinal} actually corresponds to a morphism of (symmetric coloured) topological operads which we will describe now in the case $m=2$. The topological operad $\mathcal{C}_m^{\rotboxtimes}$ is the operad with set of colours $I=\{A,B,C,D,E\}$. The operations of this operad are configurations of non-overlapping little squares inside the unit square, where each little square is coloured with an element of $I$, as in the following picture:
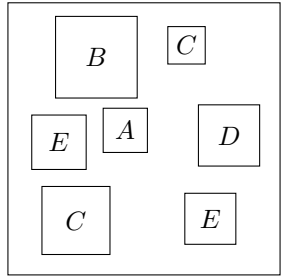
\[
\begin{tikzpicture}[scale=.9]
\draw (-2,-2) rectangle (2,2);
\draw (-1.5,-1.7) rectangle (-.5,-.7) node[midway]{$C$};
\draw (-1.65,-.45) rectangle (-.85,.35) node[midway]{$E$};
\draw (-1.3,.6) rectangle (-.1,1.8) node[midway]{$B$};
\draw (-.6,-.2) rectangle (.05,.45) node[midway]{$A$};
\draw (.6,-1.55) rectangle (1.35,-.8) node[midway]{$E$};
\draw (.35,1.1) rectangle (.9,1.65) node[midway]{$C$};
\draw (.8,-.4) rectangle (1.7,.5) node[midway]{$D$};
\end{tikzpicture}
\]
The unit square is also coloured with an element of $I$, called \emph{target colour}. There is the extra condition that there should be an arrow from the colour of any little cube to the target colour as in the following diagram:
\begin{equation}\label{equationdiamond}
	\xymatrix{
			& D \ar[d] \\
			A \ar[r] \ar[ru] \ar[rd] & E & B \ar[l] \ar[lu] \ar[ld] \\
			& C \ar[u]
	}
\end{equation}
So, if the target colour is $A$ (resp. $B$), then each little square is coloured with $A$ (resp. $B$) and if the target colour is $C$ (resp. $D$), then each little square is coloured with $A$, $B$ or $C$ (resp. $A$, $B$ or $D$). Now the topological operad $\mathcal{C}_m^\fivedots$ is the suboperad of $\mathcal{C}_m^{\rotboxtimes}$ where the configurations of little cubes satisfy extra conditions. For example, if the target colour is $E$, then the little cubes coloured with $A$ should lie in the half plane $x<0$, the ones coloured with $B$ in the half plane $x>0$, the ones coloured with $C$ on the half line $(x=0,y<0)$, the ones coloured with $D$ on the half line $(x=0,y>0)$ and there should be at most one little cube coloured with $E$, centred at the origin, as in the following picture:
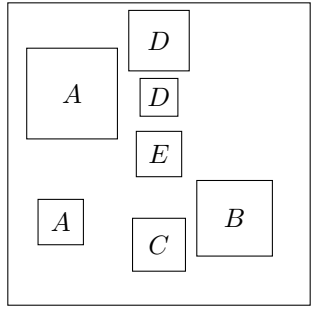
\[
\begin{tikzpicture}
\draw (-2,-2) rectangle (2,2);
\draw (-1.6,-1.2) rectangle (-1,-.6) node[midway]{$A$};
\draw (-1.75,.2) rectangle (-.55,1.4) node[midway]{$A$};
\draw (.5,-1.35) rectangle (1.5,-.35) node[midway]{$B$};
\draw (-.35,-1.55) rectangle (.35,-.85) node[midway]{$C$};
\draw (-.25,.5) rectangle (.25,1) node[midway]{$D$};
\draw (-.4,1.1) rectangle (.4,1.9) node[midway]{$D$};
\draw (-.3,-.3) rectangle (.3,.3) node[midway]{$E$};
\end{tikzpicture}
\]
It is interesting to note that, when $\mathcal{C}_m^\fivedots$ is restricted for example to the colours $A$ and $C$, one recovers Voronov's Swiss cheese operad \cite{voronov}. The morphism \eqref{equationhomotopycofinal} is the polynomial $2$-monad version of the canonical inclusion of $\mathcal{C}_m^\fivedots$ into $\mathcal{C}_m^{\rotboxtimes}$.

The paper is organised as follows. In Section \ref{sectionpreliminaries}, we will recall the notions of polynomial $2$-monads, Batanin's classifiers and Weber's formula for their computation. We will also recall some results and definitions of the homotopy theory for polynomial monads from \cite{bataninberger,batanindeleger,bwdquasitame}, which will be used in this paper. %We will introduce a notion of a \emph{homotopically cofinal} morphism of categorical operads, extending \cite[Definition 5.6]{batanindeleger} and therefore also \cite[Definition 19.6.1]{hirschhorn}.
Section \ref{sectioncofinality} is dedicated to proving the cofinality of the morphism \eqref{equationhomotopycofinal}. %, we will construct a morphism of topological operads
Finally, we will explain in Section \ref{sectiondelooping} how our cofinality result can be used to recover Ducoulombier-Turchin's delooping \eqref{equationducoulombierturchin}. The arguments are similar to the ones in \cite{deleger}.% using our cofinality result. We will first deduce a cofinality result from global to local, as we explain now. For an operad $\mathcal{P}$, let $\mathrm{Alg}(\mathcal{P})$ be its category of algebras. Let $\mathcal{C}_2^{\specialdots}$ be the operad $\mathcal{C}_2^\fivedots$ restricted to the set of colours $\{A,B,C,D\}$, and let
%\[
%	\Phi: \mathrm{Alg}(\mathcal{C}_2^{\specialdots}) \to \mathrm{CAT}
%\]
%be the functor sending a quadruple $(A,B,C,D)$ to the category of spaces $E$ equipped with the structure of a $\mathcal{C}_2^\fivedots$-algebra on $(A,B,C,D,E)$.
%
%$\mathcal{C}_m \otimes_{BV} \mathrm{Oct}_m$, $\mathfrak{C}$

\subsection*{Acknowledgement}

I would like to thank George Raptis and Christina Vasilakopoulou for our interesting discussions on this topic. I acknowledge that this work was implemented in the framework of H.F.R.I call “3rd Call for H.F.R.I.’s Research Projects to
Support Faculty Members \& Researchers” (H.F.R.I. Project Number: 23249).

\section{Preliminaries}\label{sectionpreliminaries}

%\subsection{The bicategory of polynomials}
\subsection{Polynomial $2$-monads}

Let $\mathrm{Cat}$ be the category of small categories. Recall that a functor $p: E \to B$ is \emph{exponentiable} when the \emph{base change functor} $p^*: \mathrm{Cat}/B \to \mathrm{Cat}/E$, given by taking pullbacks, has a right adjoint.

\begin{definition}\cite[Section 3.1]{weberpolynomials}
	A \emph{polynomial} in $\mathrm{Cat}$ is a diagram of shape
	\begin{equation}\label{equationpolynomial}
		\xymatrix{
			I & E \ar[l]_s \ar[r]^p & B \ar[r]^t & J
		}
	\end{equation}
	where $p$ is exponentiable. $s$ is the \emph{source map}, $p$ is the \emph{middle map} and $t$ is the \emph{target map}.
\end{definition}

%It was proved \cite[Theorem 3.1.10]{weberpolynomials} that if $\mathcal{E}$ is a category with pullbacks, we can form a bicategory.

\begin{theorem}\cite[Theorem 3.1.10]{weberpolynomials}
	There is a bicategory $\mathrm{Poly}$ whose
	\begin{itemize}
		\item objects are small categories,
		
		\item morphisms $I \to J$ are polynomials,% \eqref{equationpolynomial},
		
		\item composition of morphisms looks as follows:
		
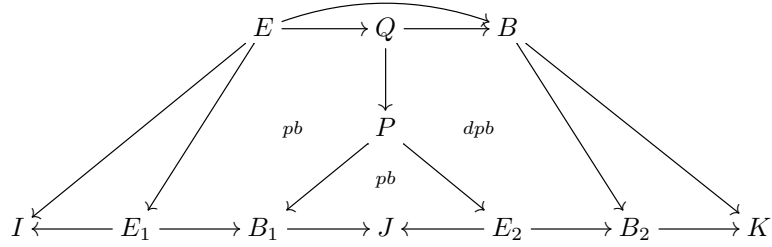
\[
		\xymatrix@C=1pc{
			&&&& E \ar[rr] \ar[lllldd] \ar[lldd] \ar@/^.8pc/[rrrr]^{} && Q \ar@{}[rrrdd]|{dpb} \ar@{}[llldd]|{pb} \ar[rr] \ar[d] && B \ar[rrrrdd] \ar[rrdd] \\
			&&&&&& P \ar@{}[d]|{pb} \ar[lld] \ar[rrd] \\
			I && E_1 \ar[rr] \ar[ll] && B_1 \ar[rr] && J && E_2 \ar[ll] \ar[rr] && B_2 \ar[rr] && K \\
			%		\cdot &&&&&&&& \cdot &&&&&&&& \cdot
		}
		\]
		where the diagram on the right is the \emph{distributivity pullback} \cite[Definition 2.2.2]{weberpolynomials},
		
		\item $2$-cells are diagrams
		\[
		\xymatrix{
			I \ar@{=}[d] & E \ar[l]_s \ar[d] \ar[r]^p \ar@{}[rd]|{pb} & B \ar[d] \ar[r]^t & J \ar@{=}[d] \\
			I & E' \ar[l]^{s'} \ar[r]_{p'} & B' \ar[r]_{t'} & J
		}
		\]
	\end{itemize}
\end{theorem}

%\subsection{Polynomial $2$-monads}

\begin{definition}
	A \emph{polynomial $2$-monad} is a monad in $\mathrm{Poly}$.
\end{definition}

%\begin{remark}
%	A \emph{polynomial monad} is a polynomial $2$-monad given by a polynomial in $\mathrm{Cat}$ which is discrete, that is, it is actually just a polynomial in $\mathrm{Set}$.
%\end{remark}

\begin{remark}%\label{remarkcartesian}
	There is a functor between bicategories \cite[Section 3.2]{weberpolynomials}
	\[
		P: \mathrm{Poly} \to \mathrm{CAT},
	\]
	where $\mathrm{CAT}$ is the $2$-category of (large) categories, functors and natural transformations. $P$ sends a category $I$ to $\mathrm{Cat}/I$ and a polynomial as in \eqref{equationpolynomial} to the composite
%	\[
%	\xymatrix{
%		\mathrm{Cat}^I \ar[r]^{s^*} & \mathrm{Cat}^E \ar[r]^{p_*} & \mathrm{Cat}^B \ar[r]^{t_!} & \mathrm{Cat}^J,
%	}
%	\]
%	where $s^*$ is the restriction functor and $p_*$ and $t_!$ are the right and left Kan extensions correspondingly.
	\[
	\xymatrix{
		\mathrm{Cat}/I \ar[r]^{s^*} & \mathrm{Cat}/E \ar[r]^{p_*} & \mathrm{Cat}/B \ar[r]^{t_!} & \mathrm{Cat}/J,
	}
	\]
	where $s^*$ is the \emph{base change} functor, $p_*$ is the \emph{dependent product} and $t_!$ is the \emph{dependent sum}, which is given by composition with $t$. In particular, any polynomial $2$-monad induces an actual $2$-monad.% Note that the induced $2$-monad is \emph{cartesian} \cite[Proposition 1.16]{gambinokock}. This means that it preserves pullbacks and the multiplication and unit are \emph{cartesian} natural transformations, that is, their naturality squares are pullbacks.
\end{remark}

\begin{example}\cite[Section 2.3]{weber}\label{examplefreemonoidmoad}
	Let $\mathbf{M}$ be the free strict monoidal category monad. For a category $A$, $\mathbf{M}(A)$ is the category whose objects are finite sequences of objects of $A$ and morphisms are levelwise maps, that is, a morphism $(a_1,\ldots,a_k) \to (b_1,\ldots,b_k)$ is given by morphisms $f_i: a_i \to b_i$, for $i=1,\ldots,k$. This monad is induced by the polynomial
	\[
		\xymatrix{
			1 & \mathbb{N}^* \ar[l] \ar[r] & \mathbb{N} \ar[r] & 1,
		}
	\]
	where $\mathbb{N}$ is the set of non-negative integers, seen as a discrete category. $\mathbb{N}^*$ is the set of pairs $(i,n)$, where $n \in \mathbb{N}$ and $1 \leq i \leq n$. The middle map forget $i$. The multiplication of the monad is of course given by concatenation. Note that, since $\mathbb{N}$ and $\mathbb{N}^*$ are discrete, this polynomial $2$-monad is actually a polynomial monad.
\end{example}

\begin{example}\cite[Section 2.3]{weber} %\label{examplefreesymmon}
	Let $\mathbf{S}$ be the free symmetric strict monoidal category monad. For a category $A$, $\mathbf{S}(A)$ is the category whose objects are finite sequences of objects of $A$ and morphisms $(a_1,\ldots,a_k) \to (b_1,\ldots,b_k)$ are given by a permutation $\rho \in \Sigma_k$ and morphisms $f_i: a_i \to b_{\rho i}$, for $i=1,\ldots,k$. This monad is induced by the polynomial
	\[
	\xymatrix{
		1 & \mathbb{P}^* \ar[l] \ar[r] & \mathbb{P} \ar[r] & 1,
	}
	\]
	where $\mathbb{P}$ is the \emph{permutation category}, whose objects are non-negative integers and the set of morphisms $\mathbb{P}(m,n)$ is the symmetric group $\Sigma_n$ if $m = n$, and is empty if $m \neq n$. $\mathbb{P}^*$ is the category of pairs $(i,n)$, where $n \in \mathbb{P}$ and $1 \leq i \leq n$. A morphism $(i,n) \to (j,n)$ is given by $\sigma \in \Sigma_n$ such that $\sigma(i) = j$. The rest of the description of the monad is as in the previous example.
\end{example}

Note that Weber \cite{weber} also considers the example of the free braided strict monoidal category monad $\mathbf{B}$.

\subsection{Internal algebra classifiers}

\begin{definition}
	An \emph{algebra} over a polynomial $2$-monad $T$ is an algebra for the induced monad on $\mathrm{Cat}/I$.
\end{definition}

\begin{definition}\cite[Definition 7.3]{batanin}
	Let $f: S \to T$ be a morphism of polynomial $2$-monads and $A$ a (strict) $T$-algebra. An \emph{internal $S$-algebra in $A$} is a lax morphism of $S$-algebras $1 \to f^*(A)$, where $1$ is the terminal $S$-algebra and $f^*$ is the restriction functor.
\end{definition}

\begin{definition}
	Let $f: S \to T$ be a morphism of polynomial $2$-monads. The \emph{internal algebra classifier} $T^S$ is the representing object of the functor
	\[
		\mathrm{Int}_S: \mathrm{Alg}_T(\mathrm{Cat}) \to \mathrm{Cat}
	\]
	sending a $T$-algebra $A$ to the category of internal $S$-algebras in $A$.
\end{definition}

Weber proved the existence and gave the formula to compute internal algebra classifiers induced by a morphism of polynomial $2$-monads, extending \cite[Theorem 7.3]{batanin}. Weber's result will be stated in Theorem \ref{theoremformulaclassifier}. Before that, we need to recall his notion of \emph{category of corners} of a \emph{crossed double category}.

%\begin{theorem}
%	
%\end{theorem}

%\subsection{Crossed double categories}

\subsection{Crossed double categories and the $2$-category of corners}

Recall that a (strict) double category is an internal category in $\mathrm{Cat}$. Explicitly, it consists of a truncated simplicial set in $\mathrm{Cat}$
\begin{equation}\label{equationdoublecategory}
	\xymatrix@C=1.5pc{
		X_1 \times_{X_0} X_1 \ar@<1.5ex>[rr]^-{p_2} \ar[rr]|-m \ar@<-1.5ex>[rr]_-{p_1} && X_1 \ar@<1.5ex>[rr]^s \ar@<-1.5ex>[rr]_t && X_0 \ar[ll]|i
	}
\end{equation}
where $m$ satisfy an extra associativity condition. $X_0$ is the category of objects and vertical morphisms, $X_1$ is the category of horizontal morphisms and squares.

%\subsection{Split opfibrations}

Recall that for a functor $f: A \to B$, a morphism $\alpha: a_1 \to a_2$ in $A$ is \emph{$f$-opcartesian} when for all $\gamma: a_1 \to a_3$ and $\beta: f(a_2) \to f(a_3)$ such that $\beta \cdot f(\alpha) = f(\gamma)$, there is a unique $\overline{\beta}: a_2 \to a_3$ such that $f(\overline{\beta}) = \beta$ and $\overline{\beta} \cdot \alpha = \gamma$:
\[
\xymatrix@R=1pc@C=2pc{
	a_1 \ar[dd]_\alpha \ar[rdd]^\gamma &&&& f(a_1) \ar[dd]_{f(\alpha)} \ar[rdd]^{f(\gamma)} \\
	&& \ar@{|->}[r]^f & \\
	a_2 \ar@{.>}[r]_-{\overline{\beta}} & a_3 &&& f(a_2) \ar[r]_\beta & f(a_3)
}
\]
A \emph{cleavage} is given by, for each $a \in A$ and $\beta: f(a) \to b$, an $f$-opcartesian lift of $\beta$, that is an $f$-opcartesian morphism $\alpha$ in $A$ such that $f(\alpha)=\beta$. The functor $f$ is a \emph{split opfibration} if there is a cleavage where the $f$-opcartesian lifts preserve the identity maps and compositions. Finally, a \emph{morphism between split opfibrations} $f_1: A_1 \to B$ and $f_2: A_2 \to B$ is given by a functor $g:A_1 \to A_2$ which preserves the opcartesian lifts and such that $f_2 g = f_1$.

%\subsection{Crossed double categories}

\begin{definition}\cite[Definition 5.1.1]{weber}\label{definitioncrosseddoublecategory}
	A double category \eqref{equationdoublecategory} 
%	\[
%	\xymatrix{
%		X_2 \ar@<1.5ex>[r]^{p_2} \ar[r]|m \ar@<-1.5ex>[r]_{p_1} & X_1 \ar@<1.5ex>[r]^s \ar@<-1.5ex>[r]_t & X_0 \ar[l]|i
%	}
%	\]
	is \emph{crossed} if $t: X_1 \to X_0$ is a split opfibration and $i$ and $m$ are morphisms of split opfibrations (note that $tm=t p_1$ is a split opfibration because split opfibrations are preserved under pullbacks and compositions):
	\[
	\xymatrix{
		X_0 \ar[r]^i \ar@{=}[rd] & X_1 \ar[d]^t & X_2 \ar[l]_m \ar[ld] \\
		& X_0
	}
	\]
\end{definition}

%\subsection{The $2$-category of corners}

\begin{definition}\cite[Definition 5.3.5]{weber}\label{definitioncategoryofcorners}
	Let $X$ be a crossed double category. The \emph{$2$-category of corners} $\mathrm{Cnr}(X)$ is the $2$-category whose
	\begin{itemize}
		\item objects are the objects of $X$,
		
		\item an arrow $x \to y$ is given by a pair $(f,g)$, where $f: x \to a$ is a vertical arrow and $g: a \to y$ is a horizontal arrow,
		
		\item composition of $(f,g): x \to y$ and $(h,k): y \to z$ is given as follows:
		\[
		\xymatrix{
			x \ar[d]^f \\
			a \ar[r]^g \ar@{.>}[d] \ar@{}[rd]|\kappa & y \ar[d]^h \\
			c \ar@{.>}[r] & b \ar[r]^k & z
		}
		\]
		where $\kappa$ is the $t$-opcartesian lift,
		
		\item a $2$-cell $(f,g) \Rightarrow (h,k)$ is given by a pair $(\alpha,\beta)$, where $\alpha$ is a vertical arrow such that $\alpha f = h$ and $\beta$ is a square as in the following diagram:
		\[
		\xymatrix{
			x \ar[d]_f \ar@/_1.3pc/[dd]_h \\
			a \ar[r]^g \ar[d]_\alpha \ar@{}[rd]|\beta & y \ar@{=}[d] \\
			b \ar[r]_k & y
		}
		\]
	\end{itemize}
\end{definition}

\subsection{Formula for internal algebra classifiers}

Let $f: S \to T$ be a morphism of polynomial $2$-monads, given by
\[
	\xymatrix{
		J \ar[d]_\phi & D \ar[l] \ar[d]_\psi \ar[r] \ar@{}[rd]|{pb} & C \ar[r] \ar[d]^\pi & J \ar[d]^\phi \\
		I \ar[r] & E \ar[r] \ar[l] & B \ar[r] & I
	}
\]
Recall that $f$ induces a square of adjunctions:
\[
	\xymatrix{
		\mathrm{Alg}_S \ar@/^/[r]^{f_!} \ar@/^/[d]^{U_S} \ar@{}[r]|\perp \ar@{}[d]|\dashv & \mathrm{Alg}_T \ar@/^/[d]^{U_T} \ar@/^/[l]^{f^*} \ar@{}[d]|\dashv \\
		\mathrm{Cat}/J \ar@/^/[r]^{\phi_!} \ar@/^/[u]^{F_S} \ar@{}[r]|\perp & \mathrm{Cat}/I \ar@/^/[l]^{\phi^*} \ar@/^/[u]^{F_T}
	}
\]
For a $2$-category $\mathbb{C}$, let $\pi_{0*} \mathbb{C}$ be the category whose objects are objects of $\mathbb{C}$ and for $X,Y \in \mathbb{C}$, the set of morphisms from $X$ to $Y$ in $\pi_{0*} \mathbb{C}$ is the set of connected components of the category of morphisms from $X$ to $Y$ in $\mathbb{C}$.

\begin{theorem}\cite{weber}\label{theoremformulaclassifier}
	Let $f: S \to T$ be a morphism of polynomial $2$-monads. Then the bar construction
	\begin{equation}\label{equationcldoublecat}
		\xymatrix{
		F_T \phi_! S^2 1 \ar@<1.5ex>[rr]^{\mu_{\phi_! S 1} \cdot F_T f_{S1}} \ar[rr]|{F_T \phi_! \mu_1} \ar@<-1.5ex>[rr]_{F_T \phi_! S!} && F_T \phi_! S 1 \ar@<1.5ex>[rr]^-{\mu_{\phi_! 1} \cdot F_T f_1} \ar@<-1.5ex>[rr]_-{F_T \phi_! !} && F_T \phi_! 1 \ar[ll]|-{F_T \phi_! \eta_1}
	}
	\end{equation}
	forms a crossed double category $X$ and $T^S$ can be computed as $\pi_{0*}\mathrm{Cnr}(X)$.%, where $X$ is the crossed double category given 
%	forms a crossed double category $\mathbb{D}$. Moreover, .%, where $\mathbb{D}$ is the crossed double category given by
%	If $X$ is a crossed double category, then the codescent object of $X$ exists and $\mathrm{CoDesc}(X) = \pi_{0*}(\mathrm{Crn}(X))$.
\end{theorem}

%\subsection{Bar construction}
%
%
%
%
%The classifier induced by $f$ can be computed as
%\begin{equation}\label{equationcldoublecat}
%	\xymatrix{
%		F_T \phi_! S^2 (1) \ar@<1.5ex>[r]^p \ar[r]|q \ar@<-1.5ex>[r]_r & F_T \phi_! S (1) \ar@<1.5ex>[r]^-d \ar@<-1.5ex>[r]_-c & F_T \phi_! (1) \ar[l]|-e
%	}
%\end{equation}
%
%\subsection{Example of computation}

\begin{example}\cite[Section 6.2]{weber}
	Let $f: \mathbf{M} \to \mathbf{S}$ be the obvious morphism given by inclusion functors. Then \eqref{equationcldoublecat} becomes
	\begin{equation}\label{equationexample}
		\xymatrix{
			\mathbf{S} \mathbf{M} (\mathbb{N}) \ar@<1.5ex>[rr]^{\mu_{\mathbb{N}} \cdot \mathbf{S}f_{\mathbb{N}}} \ar[rr]|{\mathbf{S} \mu_1} \ar@<-1.5ex>[rr]_{\mathbf{S} \mathbf{M} !} && \mathbf{S} (\mathbb{N}) \ar@<1.5ex>[rr]^-{\mu_1 \cdot \mathbf{S}f_1} \ar@<-1.5ex>[rr]_-{\mathbf{S} !} && \mathbb{P}. \ar[ll]|-{\mathbf{S} \eta_1}
		}
	\end{equation}
	The objects of this double category are non-negative integers. The category of objects and vertical morphisms is $\mathbb{P}$. The category of objects and horizontal morphisms is $\Delta_+$. The squares are given by
	\begin{equation}\label{squaredoublesm}
		\xymatrix{
			m \ar[r]^f \ar[d]_\sigma & n \ar[d]^\rho \\
			m \ar[r]_g & n
		}
	\end{equation}
	where $f$ and $g$ are order-preserving maps and $\sigma$ and $\rho$ are permutations, as illustrated below \cite[Section 4]{daystreet}:
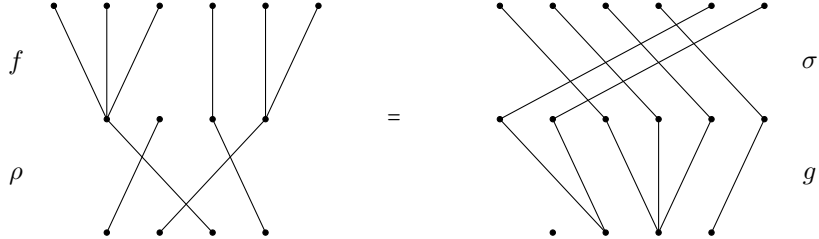
	\[
		\begin{tikzpicture}
			\draw (-.5,-.75) node{$f$};
			\draw (-.5,-2.25) node{$\rho$};
			\draw[fill] (0,0) circle (1pt);
			\draw[fill] (.7,0) circle (1pt);
			\draw[fill] (1.4,0) circle (1pt);
			\draw[fill] (2.1,0) circle (1pt);
			\draw[fill] (2.8,0) circle (1pt);
			\draw[fill] (3.5,0) circle (1pt);
			\draw[fill] (.7,-1.5) circle (1pt);
			\draw[fill] (1.4,-1.5) circle (1pt);
			\draw[fill] (2.1,-1.5) circle (1pt);
			\draw[fill] (2.8,-1.5) circle (1pt);
			\draw[fill] (.7,-3) circle (1pt);
			\draw[fill] (1.4,-3) circle (1pt);
			\draw[fill] (2.1,-3) circle (1pt);
			\draw[fill] (2.8,-3) circle (1pt);
			
			\draw (0,0) -- (.7,-1.5) -- (2.1,-3);
			\draw (.7,0) -- (.7,-1.5);
			\draw (1.4,0) -- (.7,-1.5);
			\draw (1.4,-1.5) -- (.7,-3);
			\draw (2.1,-1.5) -- (2.8,-3);
			\draw (2.1,0) -- (2.1,-1.5);
			\draw (2.8,0) -- (2.8,-1.5) -- (1.4,-3);
			\draw (3.5,0) -- (2.8,-1.5);
			
			\draw (4.5,-1.5) node{$=$};
			
			\begin{scope}[shift={(5.9,0)}]
			\draw[fill] (0,0) circle (1pt);
			\draw[fill] (.7,0) circle (1pt);
			\draw[fill] (1.4,0) circle (1pt);
			\draw[fill] (2.1,0) circle (1pt);
			\draw[fill] (2.8,0) circle (1pt);
			\draw[fill] (3.5,0) circle (1pt);
			\draw[fill] (0,-1.5) circle (1pt);
			\draw[fill] (.7,-1.5) circle (1pt);
			\draw[fill] (1.4,-1.5) circle (1pt);
			\draw[fill] (2.1,-1.5) circle (1pt);
			\draw[fill] (2.8,-1.5) circle (1pt);
			\draw[fill] (3.5,-1.5) circle (1pt);
			\draw[fill] (.7,-3) circle (1pt);
			\draw[fill] (1.4,-3) circle (1pt);
			\draw[fill] (2.1,-3) circle (1pt);
			\draw[fill] (2.8,-3) circle (1pt);
			
			\draw (0,0) -- (1.4,-1.5) -- (2.1,-3);
			\draw (.7,0) -- (2.1,-1.5) -- (2.1,-3);
			\draw (1.4,0) -- (2.8,-1.5) -- (2.1,-3);
			\draw (2.1,0) -- (3.5,-1.5) -- (2.8,-3);
			\draw (2.8,0) -- (0,-1.5) -- (1.4,-3);
			\draw (3.5,0) -- (.7,-1.5) -- (1.4,-3);
			\draw (4.1,-.75) node{$\sigma$};
			\draw (4.1,-2.25) node{$g$};
			\end{scope}
		\end{tikzpicture}
	\]
	The category of corners is $\Sigma \Delta_+$, the free symmetric strict monoidal category containing a monoid \cite{daystreet,weber}.
\end{example}

\begin{remark}
	Note that in the previous example, the functor $\mathbf{S}!$ is even a discrete opfibration. So the double category \eqref{equationexample} is \emph{codomain-discrete} \cite[Definition 2.17]{stepan}. %\v{S}t\v{e}p\'an proved that there is a correspondence between codomain-discrete double categories and categories equipped with a strict factorization system \cite[Theorem 3.8]{stepan}.
	It was proved that there is a correspondence between codomain-discrete double categories and categories equipped with a strict factorization system \cite[Theorem 3.8]{stepan}. The factorization system in $\Sigma \Delta_+$ is given by the order-preserving maps and the permutations.
\end{remark}

\subsection{Homotopy theory for polynomial monads}

%\begin{definition}
%	Let $f: S \to T$ be a morphism of categorical operads and $A$ an $S$-algebra. $f$ is \emph{homotopically left cofinal} at $A$ if the unit of the derived adjunction is a weak equivalence at $A$.
%\end{definition}

In this subsection we will recall some results and notions of homotopy theory for polynomial monads from \cite{bataninberger,batanindeleger,bwdquasitame}. We will assume that these results still hold for polynomial $2$-monads. We plan to give the proofs in \cite{delegerhomotopy}.

With sufficient assumption, the category of algebras over a polynomial monad in a monoidal category admits a \emph{transferred model structure} \cite[Theorem 2.11]{bataninberger}. Also recall that for a morphism of polynomial monads $f: S \to T$ and a $S$-algebra $X$ in a monoidal category $\mathcal{E}$, there is a functor $\widetilde{X}: T^S \to \mathcal{E}$ representing $X$ \cite[Section 6.16]{bataninberger}.

\begin{theorem}\cite[Theorem 8.2]{bataninberger}\label{theoremformulaleftderived}
	Let $\mathcal{E}$ be a monoidal model category with a ``good'' realisation functor for simplicial objects. Let $f: S \to T$ be a morphism of polynomial monads and $X$ a $S$-algebra in $\mathcal{E}$ whose underlying collection is pointwise cofibrant. Then
	\[
	\mathbb{L}f_!(X) \simeq \hocolim_{T^S} \widetilde{X},
	\]
	where $\mathbb{L}f_!$ is the left derived Quillen functor and $\widetilde{X}: T^S \to \mathcal{E}$ represents $X$.
\end{theorem}

%\begin{corollary}\cite[Corollary 8.4]{bataninberger}
%	For any morphism of polynomial monads $f: S \to T$, we have
%	\[
%		N(T^S) \simeq f_!(N(T^T)) \simeq \mathbb{L}f_!(1),
%	\]
%	where $N$ is the simplicial nerve. 
%%	where $1$ is the terminal $S$-algebra. 
%	In particular, $N(T^S)$ is cofibrant.
%\end{corollary}

\begin{definition}\cite[Definition 5.6]{batanindeleger}
	A morphism of polynomial $2$-monads $f: S \to T$ is \emph{homotopically cofinal} if $N(T^S)$ is contractible.% be a morphism of categorical operads and $A$ an $S$-algebra. $f$ is \emph{homotopically left cofinal} at $A$ if the unit of the derived adjunction is a weak equivalence at $A$.
\end{definition}

%\subsection{Tameness}

For a monad $T$ on a category $\mathbb{C}$, let $T+1$ be the monad on $\mathbb{C} \times \mathbb{C}$ given by $T \times id$, with evident multiplication and unit. If $T$ is a polynomial monad, so is $T+1$ \cite[Section 6.18]{bataninberger}, and there is a morphism of polynomial monads $T+1 \to T$ such that the restriction functor is given by the diagonal composed with the forgetful functor. Also recall that the \emph{fundamental groupoid} of a category is the groupoid obtained from this category by freely inverting all the morphisms.

\begin{definition}\cite[Definition 4.19]{bwdquasitame} %\cite[Definition 6.19]{bataninberger}
	A polynomial monad $T$ is \emph{quasi-tame} if the fundamental groupoid of the classifier $T^{T+1}$ is equivalent to a discrete groupoid.
\end{definition}

\begin{definition}\cite[Definition 1.1]{bataninberger}
	A morphism $f: X \to Y$ in a model category is an $h$-cofibration if pushouts along this morphism preserve weak equivalence. Explicitly, for all diagrams
	\[
		\xymatrix{
			X \ar[r] \ar[d]_f \ar@{}[rd]|{po} & A \ar[r]^w \ar[d] \ar@{}[rd]|{po} & B \ar[d] \\
			Y \ar[r] & A' \ar[r]_{w'} & B'
		}
	\]
	where both squares are pushouts, if $w$ is a weak equivalence, then so is $w'$.
\end{definition}

\begin{definition}\cite[Definition 1.11]{bataninberger}
	A monoidal model category is \emph{strongly $h$-monoidal} if for each cofibration (resp. weak equivalence) $f: X \to Y$ and each object $Z$, $f \otimes id: X \otimes Z \to Y \otimes Z$ is an $h$-cofibration (resp. weak equivalence).% It is \emph{strongly $h$-monoidal} if moreover the class of weak equivalences is closed under tensor product.
%	A monoidal model category is \emph{$h$-monoidal} if for each (trivial) cofibration $f: X \to Y$ and each object $Z$, $f \otimes id: X \otimes Z \to Y \otimes Z$ is a (trivial) $h$-cofibration. It is \emph{strongly $h$-monoidal} if moreover the class of weak equivalences is closed under tensor product.
\end{definition}

For a given monoidal model category, let $\mathcal{K}$ be the monoidal saturation of the class of cofibrations, that is the smallest class containing all cofibrations and which is closed under pushouts, transfinite compositions, retracts and tensoring with arbitrary objects. A monoidal model category is \emph{compactly generated} if any object is $\mathcal{K}$-small and weak equivalences are closed under filtered colimits along morphisms in $\mathcal{K}$ \cite[Definition 1.2]{bergermoerdijkhomotopy}. Also recall that a model category is \emph{left proper} if weak equivalences are preserved under pushouts along cofibrations.%, where $\mathcal{K}$ is the monoidal saturation of the class of cofibrations. 
%Recall that a class of morphisms in a monoidal category is \emph{monoidally saturated} if it is 

\begin{theorem}\cite[Theorem 4.25]{bwdquasitame}\label{theoremleftproper}
	Let $T$ be a quasi-tame polynomial monad and $\mathcal{E}$ a compactly generated strongly $h$-monoidal model category. Then the category of $T$-algebras in $\mathcal{E}$ is left proper.
\end{theorem}

\section{Cofinality result}\label{sectioncofinality}

%\subsection{Boardman-Vogt tensor product of an operad and a poset}
%
%\begin{definition}
%	Let $\mathcal{C}$ be a poset and $c \in \mathcal{C}$. A \emph{$k$-bouquet} in $\mathcal{C}$ is a $k+1$-tuple $c_1,\ldots,c_k,c$ of elements of $\mathcal{C}$ such that $c_i \leq c$ for all $i=1,\ldots,k$. $c$ is the \emph{target colour} of the $k$-bouquet.
%\end{definition}
%
%\begin{proposition}
%	Let $\mathcal{P}$ be an operad and $\mathcal{C}$ a poset, seen as an operad with only unary operations. The $k$-ary operations of $\mathcal{P} \otimes_{BV} \mathcal{C}$ are the $k$-ary operations of $\mathcal{P}$ together with $k$-bouquets in $\mathcal{C}$.
%\end{proposition}
%
%\begin{proof}
%	This can be easily checked using the explicit description of the Boardman-Vogt tensor product in terms of generators and relations.
%\end{proof}

%\subsection{The operads ${\mathcal{C}_m^{\rotboxtimes}}$}

%$\mathcal{C}_m^{\rotboxtimes}$

\subsection{The operads $\mathcal{C}_m^{\rotboxtimes}$ and $\mathcal{C}_m^\fivedots$}

\begin{definition}
	For $m \geq 0$, let $\mathfrak{D}^m$ be the poset
	\begin{equation}\label{equationposet}
		(\{-1,1\} \times \{1,\ldots,m\}) \cup \{(0,m+1)\}
	\end{equation}
	where the order is defined by $(\epsilon,l) < (\epsilon',l')$ whenever $l < l'$.% For example, $\mathrm{Oct}(2)$ looks as follows:
%	\[
%	\xymatrix{
%		& E \\
%		C \ar[ru] && D \ar[lu] \\
%		A \ar[u] \ar[rru] && B \ar[u] \ar[llu]
%	}
%	\]
\end{definition}

\begin{remark}\label{remarkoct}
	There is a canonical correspondence between the objects of $\mathfrak{D}^m$ and $2m+1$ points in $\mathbb{R}^m$. These $2m+1$ points are the centres of each of the $2m$ faces of the unit $m$-cube plus the centre of the cube. From this observation, it is easy to see that the classifying space of $\mathfrak{D}^m$ is given by the convex hull of these $2m+1$ points in $\mathbb{R}^m$. For $m=2$, we get a square as pictured in the diagram \eqref{equationdiamond} of the introduction, where $A=(-1,1)$, $B=(1,1)$, $C=(-1,2)$, $D=(1,2)$ and $E=(0,3)$. For $m=3$, we get an octahedron. The notation $\mathrm{Oct}(m)$ was in fact used for the category $\mathfrak{D}^m$ in \cite[Definition 3.1]{bataninglobular}.
\end{remark}

Recall that a \emph{coloured operad} $\mathcal{P}$ in a symmetric monoidal category $(\mathcal{V},\otimes,v)$ is given by a set of \emph{colours} $I$ and an object $\mathcal{P}(i_1,\ldots,i_k;i) \in \mathcal{V}$ of \emph{$k$-ary operations} together with
\begin{itemize}
	%	\item for $k \geq 0$ and $(c_1,\ldots,c_k,c) \in C^{k+1}$, 
	\item multiplication maps
	\begin{equation*}%\label{compositionmap}
		\mathcal{P}(i_1,\ldots,i_k;i) \otimes \mathcal{P}(i_{11},\ldots,i_{1l_1};i_1) \otimes \ldots \otimes \mathcal{P}(i_{k1},\ldots,i_{kl_k};i_k) \to \mathcal{P}(i_{11},\ldots,i_{kl_k};i),
	\end{equation*}
	\item unit maps
	\[
	v \to \mathcal{P}(i;i),
	\]
	%	called \emph{unit},
	\item symmetric group actions
	\[
	\sigma \in \Sigma_k \mapsto \left[\mathcal{P}(i_1,\ldots,i_k;i) \to \mathcal{P}(i_{\sigma^{-1}(1)},\ldots,i_{\sigma^{-1}(k)};i)\right],
	\]
\end{itemize}
satisfying associativity, unitality and equivariance axioms. When $I$ is the singleton set, we write
\[
	\mathcal{P}(k) := \mathcal{P}(\underbrace{*,\ldots,*}_k;*).
\]
A \emph{$\mathcal{P}$-algebra} in a symmetric monoidal $\mathcal{V}$-category $\mathcal{M}$ is given by an $I$-indexed collection $A := (A_i)_{i \in I}$ of objects in $\mathcal{M}$, together with maps
	\[
	\mathcal{P}(i_1,\ldots,i_k;i) \to \mathcal{M}(A_{i_1} \otimes \ldots \otimes A_{i_k},A_i),
	\]
	satisfying axioms. Finally, recall that for two coloured operad $\mathcal{P}$ and $\mathcal{Q}$, the \emph{Boardman-Vogt tensor product} $\mathcal{P} \otimes_{BV} \mathcal{Q}$ is the operad whose algebras are $\mathcal{P}$-algebras in the category of $\mathcal{Q}$-algebras, or equivalently $\mathcal{Q}$-algebras in the category of $\mathcal{P}$-algebras.% of an $I$-coloured operad $\mathcal{P}$ and a $J$-coloured operad $\mathcal{Q}$ is the $I \times J$-coloured operad $\mathcal{P} \otimes_{BV} \mathcal{Q}$ whose operations are generated by pairs $(i,q)$ where $i \in I$ and $q$ is an operation of $\mathcal{Q}$ and pairs $(p,j)$ where $p$ is an operation of $\mathcal{P}$ and $j \in J$.

\begin{definition}\label{definitioncmrotboxtimes}
	Let
	\[
		\mathcal{C}_m^{\rotboxtimes} := \mathcal{C}_m \otimes_{BV} \mathfrak{D}^m,
	\]
	where $\mathcal{C}_m$ is the \emph{little $m$-cubes operad} and the category $\mathfrak{D}^m$ is seen as an operad with only unary operations.
\end{definition}

The following is immediate:

\begin{proposition}\label{propositiondescriptionbvproduct}
	$\mathcal{C}_m^{\rotboxtimes}$ is the operad whose set of colours is the set of objects of $\mathfrak{D}^m$. The $k$-ary operations are ordered configurations of $k$ non-overlapping coloured little cubes inside the coloured unit $m$-cube. The colours satisfy the following condition. If the $i$-th little cube is coloured with $c_i \in \mathfrak{D}^m$ and the unit cube is coloured with $c \in \mathfrak{D}^m$, then $c_i \leq c$. The multiplication is given by plugging cubes of the same colour.
\end{proposition}

%\begin{proof}
%	By a direct calculation.
%%	$\mathfrak{D}^m$-algebras are covariant presheaves over $\mathfrak{D}^m$ (also called $m$-cospans \cite[Definition 3.2]{bataninglobular}). This can be checked using the explicit description of the Boardman-Vogt tensor product in terms of generators and relations.
%\end{proof}

%\subsection{The operad $\mathcal{C}_m^\fivedots$}

%\subsection{Suboperad where the little disks are nicely placed}

To a pair $(\epsilon,l) \in \{-1,1\} \times \{1,\ldots,m\}$, we associate the half-space
\[
	H(\epsilon,l) = \{ (x_1,\ldots,x_m) \in V(l) \mid \epsilon x_l > 0 \}
\]
and to $l \in \{1,\ldots,m+1\}$, we associate the subspace
\[
	V(l) = \{(x_1,\ldots,x_m) \in \mathbb{R}^m \mid x_1 = \ldots = x_{l-1} = 0\}.
\]

%\begin{definition}
%	
%\end{definition}
%
%\begin{definition}
%	For $x \in \mathcal{C}_m^{\rotboxtimes}(k)$ and $i=1,\ldots,k$, let $c_i = (\epsilon_i,l_i) \in \mathfrak{D}^m$ be the colour of the $i$-th little cube $x_i$. $x_i$ is \emph{properly placed} if its centre lies in $V(l_i)$ and if $c_i$ is not equal to the target colour, then $x_i$ lies entirely in $H(\epsilon_i,l_i)$.
%\end{definition}

\begin{definition}
	Let $\mathcal{C}_m^{\fivedots}$ be the suboperad of $\mathcal{C}_m^{\rotboxtimes}$ of operations such that for all $i=1,\ldots,k$, if the $i$-th little disk is coloured with $c_i = (\epsilon_i,l_i)$, then its centre lies in $V(l_i)$. Moreover, if $c_i$ is not equal to the target colour, then the $i$-th little disk lies entirely in $H(\epsilon_i,l_i)$.
%	
%	Let $\mathcal{C}_m$ be the following coloured operad. Its set of colours is $L_m$. Let $(c_1,\ldots,c_k,c) \in L_m^{k+1}$ such that there is a morphism $c_i \to c$ in $\mathrm{Oct}(m)$ for all $i = 1,\ldots,k$. Then
%	\[
%	\mathcal{C}_m^m(c_1,\ldots,c_k;c)
%	\]
%	is the space of non-overlapping ordered configurations of $k$ little disks inside the $m$-dimensional unit disk. The $i$-th little disk has colour $c_i$. The unit disk has target colour $c$. The extra condition is the following. Let $(\epsilon_i,l_i) := c_i$ and $(\epsilon,l) := c$. Then the centre of the $i$-th disk should lie in the space
%	\[
%	\{(x_1,\ldots,x_m) \in \mathbb{R}^m \mid \text{$x_1=\ldots=x_{l_i-1}=0$ and if $l_i \leq l$, then $\epsilon_i \cdot x_{l_i} > 0$}  \}.
%	\]
%	
%	The composition is given by plugging disks as usual.
\end{definition}

%\begin{remark}
%	Algebras over $\mathcal{C}_m^{\fivedots}$ are given by pairs $(X_l,Y_l)$ of $\mathcal{C}_l$-algebras for $l=1,\ldots,m$, plus a space $Z$, together with extra actions maps.
%%	
%%	 an $2m+1$-tuple of spaces $(X_i)_{i \in L_m}$. For $l=0,\ldots,m$, the spaces $X_{-1,l}$ and $X_{1,l}$ have an $E_l$-algebra structure. Moreover, each pair $(X_{\epsilon_1,l},X_{\epsilon_2,l+1})$ has the structure of an algebra over the $l$-dimensional Swiss cheese operad.
%\end{remark}

\subsection{The operads ${\mathcal{K}_m^{\rotboxtimes}}$ and $\mathcal{K}_m^\fivedots$}\label{subsectioncompletegraphoperad}

%\subsection{The complete graph operad}

The \emph{complete graph operad} $\mathcal{K}_m$ \cite{bergercombinatorial} is the following one-coloured operad in $\mathrm{Cat}$. For $k \geq 0$,% let $\mathbb{K}_m(k)$ be the category whose set of objects is
\[
	\mathcal{K}_m(k) := \{1,\ldots,m\}^{k \choose 2} \times \Sigma_k,
\]
where ${k \choose 2}$ is the set of pairs $ij$ with $1 \leq i < j \leq k$. It has a poset structure given by $(\mu,\sigma) \leq (\nu,\tau)$ if for all $ij \in {k \choose 2}$, $\mu_{ij} < \nu_{ij}$ or $(\mu_{ij},\sigma_{ij}) = (\nu_{ij},\tau_{ij})$, where for $\sigma \in \Sigma_k$, $\sigma_{ij} \in \Sigma_2$ is the permutation of $i$ and $j$ under $\sigma$. 
%A morphism $\mu \to \mu'$ is given by $\sigma \in \Sigma_k$ such that for all $1 \leq i < j \leq k$, if $\sigma(j) > \sigma(i)$ then $\mu_{ij} < \mu'_{ij}$. 
%
%Now\[
%\mathcal{K}_m = \coprod_{k \geq 0} \mathcal{K}_m(k),
%\]
%
%Let $\mathbb{K}_m^*(k)$ be the category of pairs $()$
%Now $\mathcal{K}_m$ is given by the polynomial
%\[
%\xymatrix{
%	1 & \mathtt{CG}^* \ar[r] \ar[l] & \mathtt{CG} \ar[r] & 1
%}
%\]
%where
%\[
%\mathbb{K}_m = \coprod_{k \geq 0} \mathbb{K}_m(k)
%\]
%and $\mathbb{K}_m^*$ is the category of elements as described in the proof of Theorem \ref{theoremfullyfaithfulfunctor}. 
The multiplication is given by
\begin{equation}\label{equationmultiplicationkm}
	\begin{aligned}
		\mathcal{K}_m(k) \times \mathcal{K}_m(n_1) \ldots \mathcal{K}_m(n_k) &\to \mathcal{K}_m(n_1+\ldots+n_k) \\
		((\mu,\sigma),(\mu_1,\sigma_1),\ldots,(\mu_k,\sigma_k)) &\mapsto (\mu(\mu_1,\ldots,\mu_k),\sigma(\sigma_1,\ldots,\sigma_k))
	\end{aligned}
\end{equation}
where $\sigma(\sigma_1,\ldots,\sigma_k)$ is given by the permutation operad and
\[
\mu(\mu_1,\ldots,\mu_k)_{ij} = 
\begin{cases}
\mu_{ij} &\text{if $i,j$ belong to different blocks of ${n_1+\ldots+n_k \choose 2}$,} \\
(\mu_s)_{ij} &\text{if $i,j$ belong to the same block ${n_s \choose 2}$ of ${n_1+\ldots+n_k \choose 2}$.}
\end{cases}
\]

%\begin{definition}
%	A \emph{graph} $G$ is given by $k \geq 0$, a permutation $\sigma \in \Sigma_k$, a set of edges $\mathsf{E} \subset \{1 \leq i < j \leq k\}$ and weights $\mu_{ij} \in \{1,\ldots,m\}$ for all $ij \in \mathsf{E}$. $k$ is the \emph{number of vertices} of $G$.
%\end{definition}
%
%\begin{definition}
%	A \emph{complete graph} is a graph whose set of edges is $\{1 \leq i < j \leq k\}$.
%\end{definition}

\begin{definition}
	Let
	\[
	\mathcal{K}_m^{\rotboxtimes} := \mathcal{K}_m \otimes_{BV} \mathfrak{D}^m.
	\]
\end{definition}

%\begin{definition}
%	A \emph{colouring} of a graph $G$ with $k$ vertices is given by $c \in \mathfrak{D}^m$ together with a $k$-bouquet with target colour $c$.
%\end{definition}

\begin{proposition}
	$\mathcal{K}_m^{\rotboxtimes}$ is the operad whose set of colours is the set of objects of $\mathfrak{D}^m$. The $k$-ary operations are elements of $\mathcal{K}_m(k)$ together with a \emph{source colour} $c_i \in \mathfrak{D}^m$ for $i=1,\ldots,k$ and a \emph{target colour} $c \in \mathfrak{D}^m$. The colours satisfy the same condition as in Proposition \ref{propositiondescriptionbvproduct}. The multiplication is as in \eqref{equationmultiplicationkm}.
\end{proposition}

%\begin{proof}
%	By a direct calculation.
%	%	$\mathfrak{D}^m$-algebras are covariant presheaves over $\mathfrak{D}^m$ (also called $m$-cospans \cite[Definition 3.2]{bataninglobular}). This can be checked using the explicit description of the Boardman-Vogt tensor product in terms of generators and relations.
%\end{proof}

For $(\mu,\sigma) \in \mathcal{K}_m(k)$ and $ij \in {k \choose 2}$, let $s_{ij} := \sgn(\sigma_{ij})$.

%\begin{definition}
%	Let $G$ be a graph with $k$ vertices and $\mathbf{c}$ be a $k$-bouquet in $\mathfrak{D}^m$. $G$ and $\mathbf{c}$ are \emph{compatible} if for each edge $ij$ of $G$,
%	\[
%	c_i \leq (-s_{ij},\mu_{ij}) \text{ or } c_j \leq (s_{ij},\mu_{ij}).%,
%	\]
%	%	where $c_1,\ldots,c_k$ is the $k$-bouquet with target colour $c$ given by the colouring and 
%\end{definition}

\begin{definition}%\label{definitionsubpolymon}
	Let $\mathcal{K}_m^{\fivedots}$ be the suboperad of $\mathcal{K}_m^{\rotboxtimes}$ of operations $(\mu,\sigma)$ such that for all $ij \in {k \choose 2}$, $c_i \leq (-s_{ij},\mu_{ij})$ or $c_j \leq (s_{ij},\mu_{ij})$.
%	\[
%		c_i \leq (-s_{ij},\mu_{ij}) \text{ or } c_j \leq (s_{ij},\mu_{ij}).%,
%	\]
%	 are good.% \geq \min(l_i,l_j)$ and if $\mu_{ij} = \min(l_i,l_j)$ then one of the following is satisfied:
	%	\begin{itemize}
		%		\item $\epsilon_i=-1$ and $l_i \leq l_j$,
		%		\item $\epsilon_j=1$ and $l_j \leq l_i$.
		%	\end{itemize}
\end{definition}

%\begin{definition}
%	The \emph{complete graph operad} $\mathcal{K}_m$ is the 
%\end{definition}

\subsection{Homotopy equivalences of operads}

Let $X$ be a topological space and $\mathcal{A}$ a poset. Recall that a \emph{cellular $\mathcal{A}$-decomposition} \cite[Definition 1.6]{bergercombinatorial} of $X$ is a collection $(X_\alpha)_{\alpha \in \mathcal{A}}$ of subspaces of $X$ % admits an \emph{$\mathcal{A}$-cellulation} if there is a functor $c: \mathcal{A} \to \mathrm{Top}$ 
satisfying the following properties:
\begin{itemize}
	\item $X_\alpha \subset X_\beta$ if and only if $\alpha \leq \beta$,
	\item $\colim_{\alpha \in \mathcal{A}} X_\alpha \simeq X$,
	\item for each $\alpha \in \mathcal{A}$, the canonical map $\colim_{\beta < \alpha} X_\beta \to X_\alpha$ is a closed fibration,
	\item for each $\alpha \in \mathcal{A}$, $X_\alpha$ is contractible.
\end{itemize}
If $X$ admits a cellular $\mathcal{A}$-decomposition, then there is a homotopy equivalence $X \simeq B\mathcal{A}$ \cite[Lemma 1.7]{bergercombinatorial}, where $B$ is the \emph{classifying space functor}, obtained by taking the geometric realization of the nerve.% Finally, 

\begin{lemma}
	There is a homotopy equivalence of topological operads $\mathcal{C}_m^{\rotboxtimes} \simeq B \mathcal{K}_m^{\rotboxtimes}$.
\end{lemma}

\begin{proof}
	The proof goes as for \cite[Theorem 1.16]{bergercombinatorial}. The argument is that for all $k \geq 0$, there is a cellular $\mathcal{K}_m^{\rotboxtimes}(k)$-decomposition of $\mathcal{C}_m^{\rotboxtimes}(k)$. For $x_1$ and $x_2$ two little cubes, we write $x_1 \square_\mu x_2$ if $x_1$ and $x_2$ are separated by a hyperplane $P_i$ perpendicular to the $i$-th coordinate axis for some $i \leq \mu$ such that, whenever there is no separating hyperplane $P_i$ for $i < \mu$, $x_1$ lies on the negative side of $P_\mu$ and $x_2$ on the positive side of $P_\mu$. For $\alpha = (\mu,\sigma) \in \mathcal{K}_m^{\rotboxtimes}(k)$, the desired cellular decomposition is given by%, for $\alpha \in \mathcal{K}_m^{\rotboxtimes}(k)$,
	\[
	\mathcal{C}_m^{\rotboxtimes}(k)_\alpha := \{x \in \mathcal{C}_m^{\rotboxtimes}(k) \mid x_i \square_{\mu_{ij}} x_j \text{ if } \sigma_{ij}=id \text{ and } x_j \square_{\mu_{ij}} x_i \text{ if } \sigma_{ij} = (12) \},
	\]
	where $x_i$ is the $i$-th little cube of $x$ for $i=1,\ldots,k$.
\end{proof}

\begin{lemma}\label{lemmaequivalenceoperads}
	There is a homotopy equivalence of topological operads $\mathcal{C}_m^\fivedots \simeq B \mathcal{K}_m^\fivedots$.
%	The canonical inclusion morphism of polynomial $2$-monads
%	\[
%		\mathcal{K}_m^\fivedots \to \mathcal{K}_m^{\rotboxtimes}
%	\]
%	is obtained from the canonical inclusion morphism of topological operad $\mathcal{C}_m^\fivedots \to \mathcal{C}_m^{\rotboxtimes}$ by applying the classifying space functor pointwise.
\end{lemma}

\begin{proof}
	For $\alpha \in \mathcal{K}_m^\fivedots(k)$, let
	\begin{equation}\label{equationcellulation}
		\mathcal{C}_m^\fivedots(k)_\alpha := \mathcal{C}_m^{\rotboxtimes}(k)_\alpha \cap \mathcal{C}_m^\fivedots(k).
	\end{equation}
	Let us prove that this space is contractible, the rest of the proof is straightforward. We will show that it is homotopy equivalent to $\mathcal{C}_m^{\rotboxtimes}(k)_\alpha$, which is contractible by the previous lemma. Let $c = (\epsilon,l)$ be the target colour of $\alpha=(\mu,\sigma)$ and $c_i = (\epsilon_i,l_i)$ be the $i$-th source colour for $i=1,\ldots,k$. For $t \in \{1,\ldots,l\}$, let $\mathcal{C}_t$ be the space of operations $x \in \mathcal{C}_m^{\rotboxtimes}(k)_\alpha$ such that for all $i=1,\ldots,k$, the centre of $x_i$ lies in $V(\lambda)$, where $\lambda = \min(t,l_i)$, and if $l_i < t$, then $x_i$ lies entirely in $H(\epsilon_i,l_i)$. 
%	\[
%		\mathcal{C}_s := \{x \in \mathcal{C}_m^{\rotboxtimes}(k)_\alpha \mid \text{the centre of $x_i$ lies in $V(\min(s,l_i))$ and if $l_i < s$ then $x_i$ lies entirely in $H(\epsilon_i,l_i)$}\}.
%	\]
	Then $\mathcal{C}_1 = \mathcal{C}_m^{\rotboxtimes}(k)_\alpha$ and $\mathcal{C}_l = \mathcal{C}_m^\fivedots(k)_\alpha$. Note that for $t \in \{1,\ldots,l-1\}$, there is a canonical inclusion map $i_t: \mathcal{C}_{t+1} \to \mathcal{C}_t$. One can also construct a map $r_t: \mathcal{C}_t \to \mathcal{C}_{t+1}$ which re-centres the little cubes with respect to the $t$-th coordinate, as in the following picture: 
%	We use the same arguments as in the proof \cite[Theorem 1.16]{bergercombinatorial}. For $c_1$ and $c_2$ two little disks, we write $c_1 \square_\mu c_2$ if $c_1$ and $c_2$ are separated by a hyperplane $H_i$ perpendicular to the $i$-th coordinate axis for some $i \leq \mu+1$ such that, whenever there is no separating hyperplane $H_i$ for $i < \mu+1$, $c_1$ lies on the negative side of $H_{\mu+1}$ and $c_2$ on the positive side of $H_{\mu+1}$. We define the cellulation $c: \mathcal{K}_m(k) \to \mathrm{Top}$ by
%	\[
%		c_\alpha := \{(d_1,\ldots,d_k) \in \mathcal{C}_m(k) \mid d_i \square_{\mu_{ij}} d_j \text{ if } \sigma(i) < \sigma(j) \text{ and } d_j \square_{\mu_{ij}} d_i \text{ if } \sigma(j) < \sigma(i) \}.
%	\]
%	The only thing we need to show is that $c_\alpha$ is the same when we take configurations of disks in the big operad or in the suboperad.
%	
	
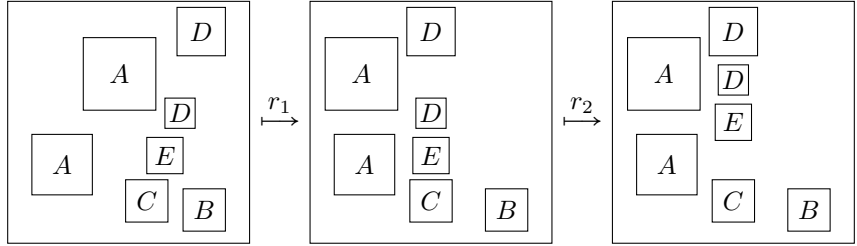
\[
	\begin{tikzpicture}[scale=.8]
		\draw (-2,-2) rectangle (2,2);
		\draw (-1.6,-1.2) rectangle (-.6,-.2) node[midway]{$A$};
		\draw (-.75,.2) rectangle (.45,1.4) node[midway]{$A$};
		\draw (.9,-1.8) rectangle (1.6,-1.1) node[midway]{$B$};
		\draw (-.05,-1.65) rectangle (.65,-.95) node[midway]{$C$};
		\draw (.6,-.1) rectangle (1.1,.4) node[midway]{$D$};
		\draw (.8,1.1) rectangle (1.6,1.9) node[midway]{$D$};
		\draw (.3,-.85) rectangle (.9,-.25) node[midway]{$E$};
		
		\draw (2.5,0) node{$\longmapsto$};
		\draw (2.5,0) node[above]{$r_1$};
		
		\begin{scope}[shift={(5,0)}]
		\draw (-2,-2) rectangle (2,2);
		\draw (-1.6,-1.2) rectangle (-.6,-.2) node[midway]{$A$};
		\draw (-1.75,.2) rectangle (-.55,1.4) node[midway]{$A$};
		\draw (.9,-1.8) rectangle (1.6,-1.1) node[midway]{$B$};
		\draw (-.35,-1.65) rectangle (.35,-.95) node[midway]{$C$};
		\draw (-.25,-.1) rectangle (.25,.4) node[midway]{$D$};
		\draw (-.4,1.1) rectangle (.4,1.9) node[midway]{$D$};
		\draw (-.3,-.85) rectangle (.3,-.25) node[midway]{$E$};
		\end{scope}

		\draw (7.5,0) node{$\longmapsto$};
		\draw (7.5,0) node[above]{$r_2$};

		\begin{scope}[shift={(10,0)}]
		\draw (-2,-2) rectangle (2,2);
		\draw (-1.6,-1.2) rectangle (-.6,-.2) node[midway]{$A$};
		\draw (-1.75,.2) rectangle (-.55,1.4) node[midway]{$A$};
		\draw (.9,-1.8) rectangle (1.6,-1.1) node[midway]{$B$};
		\draw (-.35,-1.65) rectangle (.35,-.95) node[midway]{$C$};
		\draw (-.25,.45) rectangle (.25,.95) node[midway]{$D$};
		\draw (-.4,1.1) rectangle (.4,1.9) node[midway]{$D$};
		\draw (-.3,-.3) rectangle (.3,.3) node[midway]{$E$};
		\end{scope}
	\end{tikzpicture}
	\]
%	We need to construct a deformation retract of $c_\alpha$ to when it is restricted to the suboperad. It works as follows. We recentre along one coordinate axis at a time, starting from $1$ and ending at $m$. Recentre along a coordinate $\mu$ means that we move each disk with label $(\epsilon_i,l_i)$ as follows. If $l_i > \mu$, then we move the coordinate to $0$. If $l_i = \mu$ and $\epsilon_i < 0$, we move to a negative coordinate. If $l_i = \mu$ and $\epsilon_i > 0$, we move to a positive coordinate.
	Let us prove that the map $r_t$ can be constructed. Let $x \in \mathcal{C}_t$ and $ij \in {k \choose 2}$ such that there is a line in the direction of the $t$-th coordinate which intersects both $x_i$ and $x_j$. Then $x_i$ and $x_j$ are separated by a hyperplane $P_t$ perpendicular to the $t$-th coordinate axis and there is no separating hyperplane perpendicular to another coordinate. To simplify, let us assume without loss of generality that $\sigma_{ij}=id$, the case $\sigma_{ij}=(12)$ being completely symmetric. Since $x \in \mathcal{C}_m^{\rotboxtimes}(k)_\alpha$, $x_i \square_{\mu_{ij}} x_j$, which means that $\mu_{ij} = t$ and $x_i$ lies on the negative side of $P_t$ and $x_j$ on the positive side of $P_t$. If $x_i$ needs to move to the positive side in order to be re-centred with respect to the $t$-th coordinate, then $c_i > (-1,\mu_{ij})$. Similarly, if $x_j$ needs to move to the negative side, then $c_j > (1,\mu_{ij})$. Since $\alpha \in \mathcal{K}_m^\fivedots(k)$, $c_i \leq (-1,\mu_{ij})$ or $c_j \leq (1,\mu_{ij})$. So both cubes won't need to move to the opposite side at the same time and there is no obstruction to constructing the map $r_t$ (the cubes can be made smaller if necessary). The maps $i_t$ and $r_t$ are homotopy inverses of each other, which proves that the space \eqref{equationcellulation} is contractible.
	
%	Let us assume by contradiction that for some $s \in \{1,\ldots,l-1\}$, the map $r_s$ can not be canonically defined or there is no homotopy between $i_s \cdot r_s$ and the identity map. This happens only if there is a configuration $x \in \mathcal{C}_s$ such that re-centring the little cubes $x_i$ and $x_j$ would force them to overlap. This would occur if $c_i > (1,s)$ and $c_j > (-1,s)$, which would imply that there is a hyperplane perpendicular to the $s$-th coordinate separating $x_i$ and $x_j$. Then $s \geq \mu_{ij}$ and we deduce that $c_i > (1,\mu_{ij})$ and $c_j > (-1,\mu_{ij})$. But this contradicts the fact that $\alpha$ belongs to $\mathcal{K}_m^\fivedots(k)$. In conclusion, the space $\mathcal{C}_m^\fivedots(k)_\alpha$ is indeed contractible.
%
%	The obstruction to having such a deformation retract is if for some $l=1,\ldots,m$ and two little disks $c_i$ and $c_j$, $c_i$ needs to go to the negative half-space in direction $l$ but is in the positive one, and $c_j$ needs to go to the positive but is in the negative one. This would contradict the assumption of being properly labelled. 
	In order to have $\colim_{\alpha} \mathcal{C}_m^\fivedots(k)_\alpha \simeq \mathcal{C}_m^\fivedots(k)$, we need to check that for all $x \in \mathcal{C}_m^\fivedots(k)$, there is $\alpha \in \mathcal{K}_m^\fivedots(k)$ such that $x \in \mathcal{C}_m^\fivedots(k)_\alpha$. %For $ij \in {k \choose 2}$, any two cubes $x_i$ and $x_j$, 
%	For $y$ and $z$ in $\mathbb{R}^m$, let $\mu :=  \in \{1,\ldots,m\}$ be the
%	For $ij \in {k \choose 2}$, let $\mu_{ij}$ be the minimum in $\{1,\ldots,m\}$ such that $x_i$ and $x_j$ are separated by a hyperplane perpendicular to the $\mu_{ij}$-th coordinate. Let $\sigma_{ij}=id$ if $x_i$ is on the negative side of this hyperplane and $x_j$ is on the positive side, and $\sigma_{ij}=(12)$ if $x_j$ is on the negative side and $x_i$ on the positive side. It is immediate that $x \in \mathcal{C}_m^{\rotboxtimes}(k)_\alpha$. Let us prove that $\alpha \in \mathcal{K}_m^\fivedots(k)$. For $ij \in {k \choose 2}$, the centres of both $x_i$ and $x_j$ lie in $V(\lambda)$, where $\lambda=\min(l_i,l_j)$.
	If $x \in \mathcal{C}_m^\fivedots(k)$, then $x \in \mathcal{C}_m^{\rotboxtimes}(k)$ and we know from the previous lemma that there is $\alpha \in \mathcal{K}_m^{\rotboxtimes}(k)$ such that $x \in \mathcal{C}_m^{\rotboxtimes}(k)_\alpha$. It remains to prove that $\alpha \in \mathcal{K}_m^\fivedots(k)$ automatically. Let $ij \in {k \choose 2}$ and let us prove that $c_i \leq (-s_{ij},\mu_{ij})$ or $c_j \leq (s_{ij},\mu_{ij})$. We can assume to simplify that $\sigma_{ij}=id$, without loss of generality. Since $x \in \mathcal{C}_m^\fivedots(k)$, the centres of both $x_i$ and $x_j$ lie in $V(\lambda)$, where $\lambda=\min(l_i,l_j)$. In particular there is no separating hyperplane for $t < \lambda$. Since $x_i \square_{\mu_{ij}} x_j$, we must have $\mu_{ij} \geq \lambda$. If $\mu_{ij} > \lambda$, the conclusion is immediate. If $\mu_{ij}=\lambda$, this means that $x_i$ lies on the negative side of a hyperplane perpendicular to the $\lambda$-th coordinate and $x_j$ on the positive side. Then $x_i$ lies entirely in $H(-1,\lambda)$ or $x_j$ lies entirely in $H(1,\lambda)$. Using the fact that $x \in \mathcal{C}_m^\fivedots(k)$ again, this implies that $c_i \leq (-1,\lambda)$ or $c_j \leq (1,\lambda)$.
	
	In conclusion, the conditions in the definition of $\mathcal{K}_m^{\fivedots}$ are indeed the ones necessary to have the equivalence of operads of the lemma. The remaining arguments are as in the proof of \cite[Theorem 3.5]{quesney}.
%	Note that our result also extends \cite[Theorem 3.5]{quesney}.
\end{proof}

\subsection{Corresponding polynomial $2$-monads}

%Let us describe the polynomial $2$-monad $\mathcal{K}_m$ corresponding to the \emph{complete graph operad} \cite{bergercombinatorial}.

\begin{definition}\label{definitionkmpoly2mon}
	Let $\mathcal{K}_m$ be the polynomial $2$-monad given by the polynomial
	\[
	\xymatrix{
		1 & \mathtt{CG}^* \ar[r] \ar[l] & \mathtt{CG} \ar[r] & 1,
%		1 & \mathbb{K}_m^* \ar[r] \ar[l] & \mathbb{K}_m \ar[r] & 1
	}
	\]
	where $\mathtt{CG} = \coprod_{k \geq 0} \mathtt{CG}(k)$ and $\mathtt{CG}(k)$ is the category whose set of objects is
	\[
	\{1,\ldots,m\}^{k \choose 2}
	\]
	and a morphism $\mu \to \nu$ is given by $\sigma \in \Sigma_k$ such that for all $ij \in {k \choose 2}$, if $\sigma(j) > \sigma(i)$ then $\mu_{ij} < \nu_{ij}$. Similarly, $\mathtt{CG}^* = \coprod_{k \geq 0} \mathtt{CG}^*(k)$, where $\mathtt{CG}^*(k)$ is the category whose objects are given by an object of $\mathtt{CG}(k)$ together with $i \in \{1,\ldots,k\}$ and morphisms $(\mu,i) \to (\nu,j)$ are given by a morphism $\sigma: \mu \to \nu$ such that $\sigma(i)=j$. As before, the middle map forgets $i$. Multiplication is as in \eqref{equationmultiplicationkm}.% preserve the basepoint.
\end{definition}

\begin{proposition}\label{propositionequivalencealgebras}
	The algebras of the complete graph operad $\mathcal{K}_m$ are the same as the algebras of the polynomial $2$-monad $\mathcal{K}_m$.
\end{proposition}

\begin{proof}
	The $2$-monad induced by the polynomial $2$-monad $\mathcal{K}_m$ sends a category $A$ to the category whose objects are given by $\mu \in \{1,\ldots,m\}^{k \choose 2}$ together with objects $(a_1,\ldots,a_k)$ in $A$. The morphisms are morphisms $\sigma: \mu \to \nu$ in $\mathtt{CG}(k)$ together with a morphism $a_i \to b_{\sigma i}$ for $i=1,\ldots,k$. Abusing notations, for $\sigma \in \Sigma_k$, let $\sigma: A^k \to A^k$ be the functor sending $(a_1,\ldots,a_k)$ to $(a_{\sigma 1},\ldots,a_{\sigma k})$. A strict categorical algebra of the polynomial $2$-monad $\mathcal{K}_m$ is given by a category $A$ together with a functor $\hat{\mu}: A^k \to A$ for each $\mu \in \{1,\ldots,m\}^{k \choose 2}$ and a natural transformation $\hat{\mu} \Rightarrow \hat{\nu} \cdot \sigma$ for each morphism $\sigma: \mu \to \nu$ in $\mathtt{CG}(k)$, satisfying axioms. On the other hand, an algebra of the complete graph operad is given by a category $A$, a functor $(\hat{\mu},\hat{\sigma}): A^k \to A$ for each $(\mu,\sigma) \in \mathcal{K}_m(k)$ and a natural transformation $(\hat{\mu},\hat{\sigma}) \Rightarrow (\hat{\nu},\hat{\tau})$ for each morphism $(\mu,\sigma) \to (\nu,\tau)$ in $\mathcal{K}_m(k)$, again satisfying axioms. From an algebra of $\mathcal{K}_m$ as a categorical operad, we get an algebra of $\mathcal{K}_m$ as a polynomial $2$-monad by taking $\hat{\mu} = (\hat{\mu},\hat{id})$ and $\hat{\mu} \Rightarrow \hat{\nu} \cdot \sigma$ associated to $(\mu,id) \to (\nu,\sigma)$. In the other direction, we take $(\hat{\mu},\hat{\sigma})$ as the composite $\hat{\mu} \cdot \sigma$ and $(\hat{\mu},\hat{\sigma}) \Rightarrow (\hat{\nu},\hat{\tau})$ associated to $\tau\sigma^{-1}: \mu \to \nu$. This gives us two functors between the categories of algebras which are clearly inverse of each other.
\end{proof}

\begin{remark}
	In \cite{weberoperads}, Weber constructed a $2$-monad $T/\Sigma$ from an operad $T$ in $\mathrm{Set}$. He proved that if $T$ is \emph{$\Sigma$-free} \cite[Definition 6.1]{weberoperads}, then $T/\Sigma$ is a polynomial $2$-monad. Moreover, the algebras of $T$ and $T/\Sigma$ coincide. We could extend this construction to when $T$ is an operad in $\mathrm{Cat}$. Then the polynomial $2$-monad $\mathcal{K}_m$ is actually $T/\Sigma$ when $T$ is the complete graph operad.
\end{remark}

%In view of the previous proposition, the polynomial monad $\overline{\mathcal{K}}_m$ will be denoted as $\mathcal{K}_m$, abusing notations.

\begin{definition}\label{definitionkmrotboxpolymon}
	Let $\mathcal{K}_m^{\rotboxtimes}$ be the polynomial $2$-monad given by the polynomial
	\[
	\xymatrix{
		\mathrm{ob}(\mathfrak{D}^m) & \mathtt{CCG}^* \ar[l] \ar[r] & \mathtt{CCG} \ar[r] & \mathrm{ob}(\mathfrak{D}^m)
	}
	\]
	where $\mathrm{ob}(\mathfrak{D}^m)$ is the set of objects of $\mathfrak{D}^m$. $\mathtt{CCG} = \coprod_{k \geq 0} \mathtt{CCG}(k)$ and $\mathtt{CCG}(k)$ is the category whose objects are given by $\mu \in \mathtt{CG}(k)$ equipped with $c_i \in \mathfrak{D}^m$ for $i=1,\ldots,k$ and $c \in \mathfrak{D}^m$. The morphisms $(\mu,c_1,\ldots,c_k,c) \to (\nu,d_1,\ldots,d_k,d)$ are given by morphisms $\sigma: \mu \to \nu$ in $\mathtt{CG}(k)$ such that $c_i=d_{\sigma i}$ for $i=1,\ldots,k$ and $c=d$. Again, $\mathtt{CCG}^* = \coprod_{k \geq 0} \mathtt{CCG}^*(k)$, where $\mathtt{CCG}^*(k)$ is the category whose objects are given by an object of $\mathtt{CCG}(k)$ together with $i \in \{1,\ldots,k\}$. The source map returns the colour $c_i$. The middle map forgets $i$. The target map returns the colour $c$.
\end{definition}

\begin{definition}%\label{definitionsubpolymon}
	Let $\mathcal{K}_m^{\fivedots}$ be the polynomial $2$-monad given by the polynomial
	\[
	\xymatrix{
		\mathrm{ob}(\mathfrak{D}^m) & \mathtt{PCG}^* \ar[l] \ar[r] & \mathtt{PCG} \ar[r] & \mathrm{ob}(\mathfrak{D}^m)
	}
	\]
%	\[
%	\xymatrix{
%		\mathrm{ob}(\mathfrak{D}^m) & (\mathbb{K}_m^{\fivedots})^* \ar[l] \ar[r] & \mathbb{K}_m^{\fivedots} \ar[r] & \mathrm{ob}(\mathfrak{D}^m)
%	}
%	\]
	where $\mathtt{PCG}$ is the full subcategory of $\mathtt{CCG}$ of objects $(\mu,c_1,\ldots,c_k,c) \in \mathtt{CCG}(k)$ such that for all $ij \in {k \choose 2}$, $c_i \leq (-1,\mu_{ij})$ or $c_j \leq (1,\mu_{ij})$. The rest of the description is as in Definition \ref{definitionkmrotboxpolymon}.% \geq \min(l_i,l_j)$ and if $\mu_{ij} = \min(l_i,l_j)$ then one of the following is satisfied:
	%	\begin{itemize}
	%		\item $\epsilon_i=-1$ and $l_i \leq l_j$,
	%		\item $\epsilon_j=1$ and $l_j \leq l_i$.
	%	\end{itemize}
\end{definition}

\begin{proposition}%\label{propositionequivalencealgebras}
	The algebras of the categorical operads $\mathcal{K}_m^{\rotboxtimes}$ and $\mathcal{K}_m^\fivedots$ are the same as the algebras of $\mathcal{K}_m^{\rotboxtimes}$ and $\mathcal{K}_m^\fivedots$ as polynomial $2$-monads.
\end{proposition}

\begin{proof}
	We can proceed as in the proof of Proposition \ref{propositionequivalencealgebras}.
\end{proof}

\subsection{Description of the classifier}

Let
\begin{equation}\label{equationmappolymon}
f: \mathcal{K}_m^\fivedots \to \mathcal{K}_m^{\rotboxtimes}
\end{equation}
be the morphism of polynomial $2$-monads given by inclusion functors.

%The induced double category is

%For $g: \{1,\ldots,k\} \to \{1,\ldots,l\}$ and $ij \in {l \choose 2}$, let $g^{-1}(ij)$ be the set of pairs $rs \in {k \choose 2}$ such that $g(r)=i$ and $g(s)=j$ or $g(r)=j$ and $g(s)=i$. 
For $\mu \in \mathtt{CCG}(k)$ and $S := \{l_1,\ldots,l_r\} \subset \{1,\ldots,k\}$, the restriction of $\mu$ to $S$ is defined as $\mu' \in \mathtt{CCG}(r)$ with $\mu'_{ij}:=\mu_{l_i l_j}$.

\begin{lemma}\label{lemmacomputationclassifier}
	The classifier associated to the morphism of polynomial $2$-monads \eqref{equationmappolymon} is the category whose objects are objects of $\mathtt{CCG}$ and there is a morphism $(\mu,c_1,\ldots,c_k,c) \to (\nu,d_1,\ldots,d_l,d)$ if $c=d$ and it is given by any function $g: \{1,\ldots,k\} \to \{1,\ldots,l\}$ together with morphisms $c_i \to d_{g(i)}$ in $\mathfrak{D}^m$ for $i=1,\ldots,k$ such that
	\begin{enumerate}
%		\item $c_i = d_{g(i)}$ for $i=1,\ldots,k$ and $c=d$,
		\item \label{conditionclassifier} for $ij \in {k \choose 2}$, if $g(i)<g(j)$, then $\mu_{ij}=\nu_{g(i)g(j)}$, and if $g(j)>g(i)$, then $\mu_{ij}<\nu_{g(j)g(i)}$,
		\item for $j=1,\ldots,l$, the restriction of $\mu$ to $g^{-1}(j)$ is in $\mathtt{PCG}$.
	\end{enumerate}
\end{lemma}

\begin{proof}
	The formula \eqref{equationcldoublecat} becomes
	\begin{equation}\label{equationdoublecat}
	\xymatrix{
		\mathcal{K}_m^{\rotboxtimes} \mathcal{K}_m^\fivedots (\mathtt{PCG}) \ar@<1.5ex>[r] \ar[r] \ar@<-1.5ex>[r] & \mathcal{K}_m^{\rotboxtimes} (\mathtt{PCG}) \ar@<1.5ex>[r] \ar@<-1.5ex>[r] & \mathtt{CCG} \ar[l]
	}
	\end{equation}
	The objects of the double category are the objects of $\mathtt{CCG}$. The vertical morphisms are the morphisms of $\mathtt{CCG}$. Let $\mathbb{C}$ be the category described in the statement of the lemma, which we want to prove is the classifier. The morphisms of $\mathtt{CCG}$ correspond to morphisms of $\mathbb{C}$ where $g$ is a bijection and the morphisms $c_i \to d_{g(i)}$ are the identities. The category $\mathcal{K}_m^{\rotboxtimes} (\mathtt{PCG})$ is the category whose objects are given by $\nu \in \{1,\ldots,m\}^{l \choose 2}$ together with objects $(\mu_1,\ldots,\mu_l)$ in $\mathtt{PCG}$, objects $(c_1,\ldots,c_l)$ in $\mathfrak{D}^m$ and $c \in \mathfrak{D}^m$ such that $\mu_i$ has target colour $c_i$ for $i=1,\ldots,l$. This is equivalent to giving a morphism in $\mathbb{C}$ where the function $g$ is order-preserving. Indeed, $(\mu_1,\ldots,\mu_l)$ correspond to the restrictions of $\mu$ to the fibres of $g$. As in \cite[Example 4.4.5]{weber}, a square of the double category, that is a morphism in $\mathcal{K}_m^{\rotboxtimes} (\mathtt{PCG})$, is completely determined by its boundary, and a square will exist if it forms a commutative diagram in $\mathbb{C}$. As in \cite[Example 5.2.4]{weber}, the opcartesian lifts are given by \emph{bijective-monotone} factorisation \cite[Remark 5.2.5]{weber} (see also \cite[Section 2]{bataninkockweber}). So, similarly to \cite[Example 5.3.7]{weber}, the $2$-category of corner of \eqref{equationdoublecat} is the following. The objects are objects of $\mathtt{CCG}$, the morphisms are pairs $(\rho,g)$ of morphisms in $\mathbb{C}$ where $\rho$ is a vertical morphism (given by a permutation) and $g$ is a horizontal morphism. A $2$-cell between $(\rho_1,g_1)$ and $(\rho_2,g_2)$ exists if and only if $g_2 \rho_2 \rho_1^{-1} = g_1$, that is if and only if $g_2 \rho_2 = g_1 \rho_1$.
\end{proof}

%\begin{lemma}
%	The codescent object $T^S$ is the category whose objects are given by $k \geq 0$, an element in $\{1,\ldots,m\}^{k \choose 2}$ and labels in $\mathfrak{D}^m$ for each $i = 1,\ldots,k$. Vertical morphisms are permutations $\sigma$ preserving labels and changing the order only if $\mu_{ij} < \nu_{ij}$. Horizontal morphisms are contractions of subgraphs, assuming all the edges between a vertex outside the subgraph and any vertex inside has the same weight and assuming the subgraph is properly labelled. Corners are given by contraction of subgraph $f:m \to n$ and permutations $\rho \in \Sigma_m$. Two corners $(f_1,\rho_1)$ and $(f_2,\rho_2)$ are equivalent if $f_1 \rho_1 = f_2 \rho_2$.
%\end{lemma}
%
%\begin{lemma}
%	The morphisms in the category $T^S$ are given by functions $k \to l$ in the category of finite sets, which preserve the labels and the weights.
%\end{lemma}

\subsection{Statement and proof of the cofinality result}

\begin{theorem}\label{theoremcofinality}
	The morphism of polynomial $2$-monads \eqref{equationmappolymon} is homotopically cofinal.% at $\mathcal{K}_m$.
	%	\[
	%		f: \mathcal{K}_m^\fivedots \to \mathcal{K}_m^{\rotboxtimes}
	%	\]
	%	be the inclusion map of operads. Let
	%	\[
	%		\xymatrix{
	%			\mathcal{K}_m^\fivedots\text{-}\mathrm{Alg}(M) \ar[r] & \mathcal{K}_m^{\rotboxtimes}\text{-}\mathrm{Alg}_{}(M)
	%		}
	%	\]
	%	be the Quillen adjunction induced by $f$. Then there is a weak equivalence of $\mathcal{K}_m^{\rotboxtimes}$-algebras in $M$
	%	\[
	%		Lf_! (\mathcal{K}_m) \sim \mathcal{K}_m,
	%	\]
	%	where $Lf_!$ is the left derived functor of $f_!$.
\end{theorem}

	We will proceed as in \cite[Section 10]{batanindeleger}. We have the commutative square of polynomial $2$-monads%categorical operads
	\begin{equation}\label{equationsquare}
	\xymatrix{
		\mathcal{K}_m^\fivedots \ar[r]^{pf} \ar[d]_f & \mathcal{K}_m \ar@{=}[d] \\
		\mathcal{K}_m^{\rotboxtimes} \ar[r]_p & \mathcal{K}_m
	}
	\end{equation}
	%	\[
	%		\xymatrix{
	%			S \ar[r]^{pf} \ar[d]_f & T_0 \ar@{=}[d] \\
	%			T \ar[r]_p & T_0
	%		}
	%	\]
	According to \cite[Proposition 4.7]{batanindeleger}, this induces a functor
	\begin{equation}\label{equationsmoothfunctor}
	(\mathcal{K}_m^{\rotboxtimes})^{\mathcal{K}_m^\fivedots} \to p^*\left((\mathcal{K}_m)^{\mathcal{K}_m}\right).
	\end{equation}
	Recall \cite[Section 5.3.1]{cisinski} that a functor $F$ is \emph{smooth} if for all $y \in \mathcal{Y}$, the canonical inclusion
	\begin{equation}\label{equationcanonicalinclusion}
		F_y \to y/F,
	\end{equation}
	where $F_y$ is the fibre over $y$, induces a weak equivalence between nerves.
	
	\begin{lemma}\label{lemmasmooth}
		The functor \eqref{equationsmoothfunctor} is smooth.
	\end{lemma}

	\begin{proof}
		To simplify the notations, let us denote the functor \eqref{equationsmoothfunctor} by $F: \mathcal{X} \to \mathcal{Y}$. For $g_1: y \to y_1$ in $\mathcal{Y}$ and $x_1 \in \mathcal{X}$ such that $F(x_1) = y_1$, let $\mathcal{X}(x_1,g_1)$ be the category whose objects are morphisms $g: x \to x_1$ such that $F(g)=g_1$ and a morphism from $g:x \to x_1$ to $g':x' \to x_1$ is given by $h: x \to x'$ in $F_y$ such that $g'h=g$. According to \cite[Proposition 5.3.4]{cisinski}, $F$ is smooth if $\mathcal{X}(x_1,g_1)$ has contractible nerve for all $x_1$ and $g_1$.
		
		First let us describe $F$ explicitly. Note that the category $\mathcal{Y}$ admits a description similar to the description of $\mathcal{X}$ given in Lemma \ref{lemmacomputationclassifier}. The objects are elements of $\mathtt{CG}$, so they are given by $k \geq 0$ and $\mu \in \{1,\ldots,m\}^{k \choose 2}$. The morphisms $\mu \to \nu$ are given by functions $g: \{1,\ldots,k\} \to \{1,\ldots,l\}$ such that for $ij \in {k \choose 2}$, if $g(i) < g(j)$, then $\mu_{ij}=\nu_{g(i)g(j)}$, and if $g(j) > g(i)$, then $\mu_{ij}<\nu_{g(j)g(i)}$. Note that there is a terminal object given by $k=1$. The functor $F$ just forgets the colours $c_i \in \mathfrak{D}^m$ for $i=1,\ldots,k$ and $c \in \mathfrak{D}^m$.
		
		Let us now prove that $\mathcal{X}(x_1,g_1)$ has contractible nerve for all $x_1$ and $g_1$ as above. The objects of this category are given by $x \in F_y$ together with $g: x \to x_1$ such that $F(g)=g_1$. Note that $g$ is entirely determined by $g_1$ and is given by a function $\{1,\ldots,k\} \to \{1,\ldots,l\}$, so $\mathcal{X}(x_1,g_1)$ can be seen as a subcategory of $F_y$. Moreover, objects of $F_y$ are given by the object $y \in \mathcal{Y}$, that is $y \in \mathtt{CG}(k)$ for some $k \geq 0$, together with a colouring of $y$, that is $c_i \in \mathfrak{D}^m$ for $i=1,\ldots,k$ and $c \in \mathfrak{D}^m$. So $\mathcal{X}(x_1,g_1)$ is the category whose objects $x$ are given by colourings of $y$ such that for $j=1,\ldots,l$, the restriction of $x$ to the $j$-th fibre of $g$ gives an element of $\mathtt{PCG}$. We can therefore assume without loss of generality that $y_1$ is the terminal object. Indeed, the category $\mathcal{X}(x_1,g_1)$ in the general case will be isomorphic to the product over the fibres of $g_1$ of $\mathcal{X}(x_1^j,g_1^j)$, where $x_1^j$ and $g_1^j$ are the restrictions of $x_1$ and $g_1$ to the $j$-th fibre of $g_1$, so $g_1^j$ is just the unique map to the terminal object. When $y_1$ is the terminal object, $x_1$ is given by $d_1 \in \mathfrak{D}^m$ and $d \in \mathfrak{D}^m$. $\mathcal{X}(x_1,g_1)$ is the category whose objects are given by $c_i \in \mathfrak{D}^m$ such that $c_i \leq d_1$ for $i=1,\ldots,k$ and the obtained colouring of $y$ is an element of $\mathtt{PCG}$. The morphisms are just given by morphisms $c_i \to c'_i$ in $\mathfrak{D}^m$ for $i=1,\ldots,k$.
		
		Since $\mathcal{X}(x_1,g_1)$ depends on $y \in \mathtt{CG}(k)$ and $d_1 \in \mathfrak{D}^m$, let $\chi(d_1,y) := \mathcal{X}(x_1,g_1)$ (using the notation in the proof of \cite[Lemma 3.16]{deleger}). Let us now prove that the nerve of $\chi(d_1,y)$ is contractible. We will prove a more general statement. Let $S \subset {k \choose 2}$, $y \in \{1,\ldots,m\}^S$ and $d:=(d_1,\ldots,d_k)$, where $d_i=(\epsilon_i,l_i) \in \mathfrak{D}^m$ for $i=1,\ldots,k$. Let us assume that for all $ij \in S$, $\epsilon_i \leq \epsilon_j$. Let $\chi(d,y)$ be the category whose objects are given by $c_i \in \mathfrak{D}^m$ such that $c_i \leq d_i$ for $i=1,\ldots,k$ and for $ij \in S$, $c_i \leq (-1,y_{ij})$ or $c_j \leq (1,y_{ij})$. As above, the morphisms are given by morphisms $c_i \to c'_i$ in $\mathfrak{D}^m$ for $i=1,\ldots,k$. We will prove that $\chi(d,y)$ has contractible nerve.
		
		Let us proceed by induction on the number of elements in $S$. The case where $S$ is empty is trivial. If $S$ is non-empty, let $ij \in S$. %We can assume without loss of generality that if $\epsilon_i=\epsilon_j=-1$, then $j$ is maximal with this property, that is for all $i'j' \in S$ with $\epsilon_{i'}=\epsilon_{j'}=-1$, $j' \leq j$. Similarly, we can assume that if $\epsilon_i=\epsilon_j=1$, then $i$ is minimal with this property. 
		Let us assume that $\epsilon_i \neq 1$, otherwise $\epsilon_j \neq -1$ and we can apply a similar argument. Let $d'_r := d_r$ for $r=1,\ldots,k$ with $r \neq i$ and let $d'_i := (-1,\lambda)$, where $\lambda=\min(l_i,y_{ij})$. %Let $S'$ be the set $S$ without $ij$ and $y'$ the restriction of $y$ to $S'$.
		There is a functor $\alpha: \chi(d,y) \to \chi(d',y)$ which replaces $c_i$ by $d'_i$ if $c_i \nless d'_i$. Let us prove that $\alpha$ induces a weak equivalence between nerves. For $c' \in \chi(d',y)$, the fibre category $\alpha_{c'}$ is trivial if $c'_i \neq d'_i$. If $c'_i = d'_i$, it is equivalent to the category whose objects are $c_i \in \mathfrak{D}^m$ such that $c_i \leq d_i$ but $c_i \nless d'_i$. This last category has a terminal object given by $d_i$. Moreover the canonical inclusion $\iota: \alpha_{c'} \to c'/\alpha$ has a right adjoint. This right adjoint sends $c' \to \alpha(c)$ to $c''$ given by $c''_r := c'_r$ for $r=1,\ldots,k$ with $r \neq i$ and $c''_i := c'_i$ if $c'_i < d'_i$ and $c''_i := c_i$ otherwise. The unit is the identity, the counit is the unique map. So $\iota$ induces a weak equivalence between nerves, which means that $c'/\alpha$ has a contractible nerve for all $c'$. We deduce, using Quillen's Theorem A, that $\alpha$ induces a weak equivalence between nerves. Note that the category $\chi(d',y)$ is equivalent to $\chi(d',y')$, where $y'$ is the restriction of $y$ to $S'$, and $S'$ is the set $S$ minus $ij$. By induction the nerve $\chi(d',y')$ is contractible, so the nerve of $\chi(d,y)$ is also contractible. This concludes the proof.
%		Let $\chi_i(d,y)$ be the full subcategory of $\chi(d,y)$ of objects such that $c_i \leq (-1,y_{ij})$. Similarly, let $\chi_j(d,y)$ be the full subcategory of $\chi(d,y)$ of objects such that $c_j \leq (1,y_{ij})$. We want to prove that $\chi_i(d,y)$, $\chi_j(d,y)$ and $\chi_i(d,y) \cap \chi_j(d,y)$ have contractible nerve.
%		
%		
		%
		%In conclusion, we need to prove that, given $y \in \mathtt{CG}(k)$ and $d_1 \in \mathfrak{D}^m$,
	\end{proof}
	
\begin{proof}[Proof of Theorem \ref{theoremcofinality}]
	We want to prove that $\mathcal{X}$ has a contractible nerve. As it was pointed out in the proof of Lemma \ref{lemmasmooth}, $\mathcal{Y}$ has a contractible nerve since it has a terminal object. It remains to prove that $F$ induces a weak equivalence between nerves.
	
	Let us prove that for all $y \in \mathcal{Y}$, $F_y$ has a contractible nerve. To be precise, $\mathcal{Y}$ is a collection of categories indexed by the objects of $\mathfrak{D}^m$. Let $y \in \mathcal{Y}$ indexed by $c \in \mathfrak{D}^m$. If $y \in \mathtt{CG}(k)$, the fibre category $F_y$ is the category whose objects are given by $c_i \in \mathfrak{D}^m$ such that $c_i \leq c$ for $i=1,\ldots,k$. The morphisms are given by morphims $c_i \to c_{i'}$ in $\mathfrak{D}^m$ for $i=1,\ldots,k$. This category has an obvious terminal object given by $c_i=c$ for all $i$. So $F_y$ indeed has a contractible nerve.
	
	According to Lemma \ref{lemmasmooth}, $F$ is smooth, which means by definition that the functor \eqref{equationcanonicalinclusion} induces a weak equivalence between nerve. So $y/F$ has a contractible nerve for all $y \in \mathcal{Y}$. The conclusion that $F$ induces a weak equivalence between nerves follows from Quillen's Theorem A.
%	It remains to prove that $F$ is a weak equivalence. Recall \cite[Section 5.3.1]{cisinski} that functor $F$ is \emph{smooth} if for all $y \in \mathcal{Y}$, the canonical inclusion
%	\[
%		F_y \to y/F,
%	\]
%	where $F_y$ is the fibre over $y$, is a weak equivalence. According to \cite[Proposition 5.3.4]{cisinski}, $F$ is smooth if for all $f: y_0 \to y_1$ in $\mathcal{Y}$ and $x_1 \in \mathcal{X}$ such that $F(x_1) = y_1$, the \emph{lifting category} whose objects are 
%	
%	The right side is contractible. The fibres are contractible. Indeed, they are given by the product of $P$. Now let us prove that the functor is smooth. The lifting categories are given by $\mathcal{C}(x)$. This concludes the proof. $\mathcal{Y}$ has a contractible nerve since it has a terminal object.
\end{proof}

\section{Delooping of mapping spaces}\label{sectiondelooping}

%\subsection{Bimodules as algebras in the category of right modules}
%
%\begin{definition}
%	For an operad $\mathcal{P}$, let $\mathcal{P}\text{-}\mathrm{RMod}$ be the $2$-category of right $\mathcal{P}$-modules. It is a symmetric monoidal $2$-category \cite[Section 6.1]{fressebook} with monoidal product given by
%	\[
%	(A \otimes B)_n = \coprod_{p+q=n} (A(p) \times B(q))_{\Sigma_p \times \Sigma_q},
%	\]
%	and unit $I = (1,0,0,\ldots)$.
%\end{definition}
%
%Note that algebras over $\mathcal{K}_m$ in $\mathcal{K}_m\text{-}\mathrm{RMod}$ are equivalent to $\mathcal{K}_m$-bimodules \cite[Proposition 9.1.2]{fressebook}. In particular, since $\mathcal{K}_m$ is a bimodule over itself, it is also an algebra over itself in $\mathcal{K}_m\text{-}\mathrm{RMod}$. Abusing notation, we will write $\mathcal{K}_m$ for the restriction of $\mathcal{K}_m$ through the composite
%\begin{equation*}%\label{equationcomposite}
%	\mathcal{K}_m^\fivedots \xrightarrow{f} \mathcal{K}_m^{\rotboxtimes} \xrightarrow{p} \mathcal{K}_m,
%\end{equation*}
%where $f$ in the inclusion and $p$ is the projection.

\subsection{Left properness}% for $\mathcal{K}_m$-bimodules}

\begin{lemma}\label{lemmaquasitame}
	The polynomial $2$-monad $\mathcal{K}_m$ is quasi-tame.
\end{lemma}

\begin{proof}
	We need to compute the classifier $T^{T+1}$ when $T=\mathcal{K}_m$. Using the description given in \cite[Section 6.20]{bataninberger}, we can prove, as in Lemma \ref{lemmacomputationclassifier}, that we get the following category. The objects are given by $k \geq 0$ and an object of $\mathtt{CG}(k)$, together with a colour $X$ or $K$ for $i=1,\ldots,k$. The morphisms are given by functions $g: \{1,\ldots,k\} \to \{1,\ldots,l\}$ satisfying the condition \ref{conditionclassifier} of Lemma \ref{lemmacomputationclassifier}. Moreover, the function $g$ should send an element $i \in \{1,\ldots,k\}$ coloured with $X$ or $K$ to an element of $\{1,\ldots,l\}$ with the same colour, and $g$ restricted to the set of elements coloured with $K$ should be injective and order-preserving. This category is a coproduct of categories with an initial object. The initial object in each component is given the object where each $i \in \{1,\ldots,k\}$ is coloured with $K$. In particular, the fundamental groupoid of this category is equivalent to a discrete groupoid.% and prove that it is a coproduct of categories with terminal object. The terminal object in each connected component is given by...
\end{proof}

\begin{remark}
	It was proved in \cite[Section 9.2]{bataninberger} that the polynomial monad $\mathbf{M}$ of example \ref{examplefreemonoidmoad} is \emph{tame} \cite[Definition 6.19]{bataninberger}. More explicitly, it was proved that the classifier $\mathbf{M}^{\mathbf{M}+1}$ is a coproduct of categories with a terminal object. In fact, each connected component has both an initial and a terminal object. In the proof of the lemma above however, the connected components of the classifier have an initial but not always a terminal object.
\end{remark}

%\begin{lemma}
%	The category of right modules over an operad is strongly $h$-monoidal.
%\end{lemma}
%
%\begin{proof}
%	
%\end{proof}

Let $\mathcal{A}$ and $\mathcal{B}$ be two (one-coloured) operads in a symmetric monoidal category $(\mathcal{V},\otimes,v)$. Recall that a \emph{left $\mathcal{A}$-module} is given by a collection $\mathcal{C} := (\mathcal{C}(k))_{k \geq 0}$ of objects in $\mathcal{V}$ together with maps
\[
	\mathcal{A}(k) \otimes \mathcal{C}(n_1) \otimes \ldots \otimes \mathcal{C}(n_k) \to \mathcal{C}(n_1+\ldots+n_k)
\]
satisfying axioms. Similarly, a \emph{right $\mathcal{B}$-module} is given by a collection $\mathcal{C} := (\mathcal{C}(k))_{k \geq 0}$ of objects in $\mathcal{V}$ together with maps
\[
\mathcal{C}(k) \otimes \mathcal{B}(n_1) \otimes \ldots \otimes \mathcal{B}(n_k) \to \mathcal{C}(n_1+\ldots+n_k)
\]
satisfying axioms. Finally, $\mathcal{C}$ is an \emph{$\mathcal{A}-\mathcal{B}$-bimodule} if it has left $\mathcal{A}$-module and right $\mathcal{B}$-module structure satisfying an extra compatibility axiom. An  \emph{$\mathcal{A}$-bimodule} is an $\mathcal{A}-\mathcal{B}$-bimodule where $\mathcal{A}=\mathcal{B}$.

\begin{lemma}\label{lemmaleftproper}
	The category of $\mathcal{K}_m$-bimodules is left proper.
\end{lemma}

\begin{proof}
	First note that the category of bimodules over an operad $\mathcal{P}$ is equivalent to the category of algebras in the category of right modules over this operad \cite[Proposition 9.1.2]{fressebook}. The monoidal product $\otimes$ in the category of right modules is given by convolution. So we need to apply Theorem \ref{theoremleftproper} to the case where the polynomial $2$-monad $T$ is $\mathcal{K}_m$ and $\mathcal{E}$ is the category of right $\mathcal{K}_m$-modules. We know from Lemma \ref{lemmaquasitame} that $\mathcal{K}_m$ is quasi-tame. It remains to check that the category of right $\mathcal{K}_m$-modules is strongly $h$-monoidal. Since it is a category of presheaves, any cofibration $f$ is in particular a pointwise cofibration. Therefore, $f \otimes id$ is still a pointwise cofibration. This means that pushouts along $f \otimes id$ preserve weak equivalences, since pushouts and weak equivalences are given pointwise and $\mathrm{Cat}$ is left proper. So the category of right $\mathcal{K}_m$-modules is indeed strongly $h$-monoidal.
	
	Note that an alternative proof could have been to show directly that the polynomial $2$-monad for $\mathcal{K}_m$-bimodules is quasi-tame.
\end{proof}

\subsection{Delooping}

\begin{definition}\cite[Definition 4.2]{deleger}
	Let $\mathcal{M}$ be a model category and $A \in \mathcal{M}$. For $m \geq 0$, let $S^{m-1}(A) \in \mathcal{M}$ defined by induction as follows. $S^{-1}(A)$ is the initial object and, for $m \geq 0$, $S^m(A)$ is the pushout
	\[
	\xymatrix{
		S^{m-1}(A) \ar@{>->}[r] \ar@{>->}[d] \ar@{}[rd]|{po} & D^m(A) \ar[d] \\
		D^m(A) \ar[r] & S^m(A)
	}
	\]
	where $D^m(A)$ is a factorisation
	\[
	\xymatrix@C=1pc{
		S^{m-1}(A) \ar[rr] \ar@{>->}[rd] && A \\
		& D^m(A) \ar[ru]_\sim
	}
	\]
\end{definition}

To simplify the notations, let $t=(0,m+1)$ be the terminal object of $\mathfrak{D}^m$.
		
\begin{definition}
	Let $\mathfrak{S}^{m-1}$ be the full subcategory of $\mathfrak{D}^m$ of all the objects of $\mathfrak{D}^m$ except $t$.
\end{definition}

	\begin{lemma}\label{lemmaspherehocolimformula}
		Let $\mathcal{M}$ be a model category. For $A \in \mathcal{M}$, $S^{m-1}(A)$ can be computed as the homotopy colimit of the functor $\delta_A: \mathfrak{S}^{m-1} \to \mathcal{M}$ which is constant to $A$.
	\end{lemma}
	
	\begin{proof}
		Let $\delta'_A: \mathfrak{S}^{m-1} \to \mathcal{M}$ be the functor sending $(\epsilon,l)$ to $D^l(A)$. Then $\delta'_A$ is a cofibrant replacement $\delta_A$ in the category of covariant presheaves over $\mathfrak{S}^{m-1}$ in $\mathcal{M}$. Therefore $\hocolim \delta_A = \colim \delta'_A = S^{m-1}(A)$.
	\end{proof}
	
	For a model category $\mathcal{M}$, let $\mathcal{M}^\mathsf{h}(-,-)$ be the derived mapping space in $\mathcal{M}$. Let $\mathrm{Bimod}_{\mathcal{K}_m}$ be the category of $\mathcal{K}_m$-bimodules.
	
	\begin{lemma}\label{lemmageneraldelooping}
		For any morphism $f: X \to Y$ in $\mathrm{Bimod}_{\mathcal{K}_m}$, there is a weak equivalence
		\[
		\Omega^m \mathrm{Bimod}_{\mathcal{K}_m}^\mathsf{h}(X,Y) \xrightarrow{\sim} (S^{m-1}(X)/\mathrm{Bimod}_{\mathcal{K}_m})^\mathsf{h} (X,Y).
		\]
%		where, for a model category $\mathcal{M}$, $\mathcal{M}^\mathsf{h}(-,-)$ is the derived mapping space in $\mathcal{M}$.
	\end{lemma}
	
	\begin{proof}
		This delooping actually holds not only for $\mathrm{Bimod}_{\mathcal{K}_m}$ but in general for any left proper model category. This was proved in \cite[Theorem 4.5]{deleger} when $X$ is contractible. However the theorem actually remains true for any $X$, since the same arguments still apply. Since $\mathrm{Bimod}_{\mathcal{K}_m}$ is left proper according to Lemma \ref{lemmaleftproper}, it also holds in this particular case.
	\end{proof}

%	\begin{theorem}
%		Let $T$ be a categorical operad and $A \to B$ be a strict morphism of categorical $T$-algebras. If $T$ is quasi-tame, we have a weak equivalence
%		\[
%		\Omega \mathrm{Alg}_T (A,B) \to \mathrm{Alg}_{T_A} (A,B)
%		\]
%	\end{theorem}
%	
%	\begin{proof}
%		It is a generalisation of \cite[Theorem 4.5]{deleger}.
%	\end{proof}

\subsection{Applying the cofinality result}

Let $\mathcal{M}$ be the category of right $\mathcal{K}_m$-modules.

\begin{definition}\label{definitionphi}
	Let $\mathcal{K}_m^{\diamonddots}$ be the polynomial $2$-monad $\mathcal{K}_m^{\fivedots}$ restricted to the set of objects of $\mathfrak{S}^{m-1}$. Let $\mathbb{B}$ be the category of algebras over $\mathcal{K}_m^{\diamonddots}$ and
	\[
	\Phi: \mathbb{B} \to \mathrm{CAT}
	\]
	be the contravariant functor which sends $b \in \mathbb{B}$ to the category of objects $X \in \mathcal{M}$ together with the structure of a $\mathcal{K}_m^{\fivedots}$-algebra on the pair $(b,X)$.
\end{definition}

\begin{definition}
	Let $\mathcal{K}_m^{\rotbox}$ be the polynomial $2$-monad $\mathcal{K}_m^{\rotboxtimes}$ restricted to the set of objects of $\mathfrak{S}^{m-1}$. Let $\mathbb{C}$ be the category of algebras over $\mathcal{K}_m^{\rotbox}$ and
	\[
	\Psi: \mathbb{C} \to \mathrm{CAT}
	\]
	be the contravariant functor which sends $c \in \mathbb{C}$ to the category of objects $Y \in \mathcal{M}$ together with the structure of a $\mathcal{K}_m^{\rotboxtimes}$-algebra on the pair $(c,Y)$.
\end{definition}

%Test $\square$ and $\boxtimes$.

%\begin{definition}
%	Let $\mathcal{B}$ and $\mathcal{C}$ be respectively the categories of algebras of $\mathcal{K}_m^{\diamonddots}$ and $\mathcal{K}_m^{\rotbox}$. Let
%	\[
%		\Phi: \mathcal{B}^{op} \to \mathrm{CAT}
%	\]
%	be the functor which sends $b \in \mathcal{B}$ to the category of small categories $X$ such that the pair $(b,X)$ has the structure of a $\mathcal{K}_m^{\fivedots}$-algebra.
%\end{definition}
%
%\begin{lemma}
%	If $f: S \to T$ is homotopically cofinal, then for any $T$-algebra $A$, $f_! f^*(A) \sim A$.
%\end{lemma}
%
%\begin{proof}
%	Let $\tilde{A}: T^T \to M$ be the functor representing $A$. We have
%	\[
%		\mathbb{L}f_! f^*(A) = \hocolim_{T^S} f^*(A) = \hocolim_{T^T} A = A.
%	\]
%\end{proof}
%To simplify the notations, let $(b) := D^0(\mathcal{K}_m)$.

\begin{lemma}\label{lemmareflection}
	Let $f: S \to T$ be a morphism of polynomial $2$-monads. If the functor $T^S \to T^T$ preserves homotopy colimits, then for $T$-algebra $X$, there is a weak equivalence
	\[
		\mathbb{L} f_! f^*(X) \simeq X.
	\]
\end{lemma}

\begin{proof}
	It is a direct consequence of Theorem \ref{theoremformulaleftderived}.
\end{proof}

%\begin{lemma}
%	There is a weak equivalence $\mathbb{L}f_! f^* (\mathcal{K}_m) \simeq \mathcal{K}_m$.
%\end{lemma}
%
%\begin{proof}
%	According to Theorem \ref{theoremformulaleftderived}, $\mathbb{L}f_! (\mathcal{K}_m)$ can be computed as a homotopy colimit over the classifier $\mathcal{X}$ associated to $f$. As it was proved in Theorem \ref{theoremcofinality}, the functor $F: \mathcal{X} \to \mathcal{Y}$ is such that for all $y \in \mathcal{Y}$, $y/F$ has a contractible nerve. This means that $F$ preserves homotopy colimits. 
%\end{proof}

Let $f: \mathcal{K}_m^\fivedots \to \mathcal{K}_m^{\rotboxtimes}$ and $p: \mathcal{K}_m^{\rotboxtimes} \to \mathcal{K}_m$ be the morphisms of polynomial $2$-monads as in the square \eqref{equationsquare}. Abusing notations, let $\mathcal{K}_m$ be the algebra over itself in $\mathcal{M}$ and $\mathcal{K}_m := p^*(\mathcal{K}_m)$ and $\mathcal{K}_m := f^* p^* (\mathcal{K}_m)$.%, where $\mathcal{K}_m$ is seen as an algebra over itself in $\mathcal{M}$.%, to  := f^* p^* (\mathcal{K}_m)$, where $\mathcal{K}_m$ is .%, along the map of polynomial $2$-monads $pf$ of the square \eqref{equationsquare}.

\begin{lemma}\label{lemmaquillenreflection}
	If $(b,X) = D^0(\mathcal{K}_m) \in \int \Phi$, then for any map $\mathcal{K}_m \to Y$ in $\Psi(b)$, there is a weak equivalence
	\[
		\Psi(b)^\mathsf{h} (\mathcal{K}_m,Y) \xrightarrow{\sim} \Phi(b)^\mathsf{h} (\mathcal{K}_m,Y).
	\]
%	There is a Quillen reflection
%	\[
%	\xymatrix{
%		\Phi(D^0(\mathcal{K}_m)) \ar@/^/[r] \ar@{}[r]|\perp & \Psi(D^0(\mathcal{K}_m)) \ar@/^/[l]
%	}
%	\]
\end{lemma}

\begin{proof}
	We follow the same arguments as in the proof of \cite[Lemma 4.11]{deleger}. The map of polynomial $2$-monads $f: \mathcal{K}_m^{\fivedots} \to \mathcal{K}_m^{\rotboxtimes}$ restricts to a map $g: \mathcal{K}_m^{\diamonddots} \to \mathcal{K}_m^{\rotbox}$. It is also homotopically cofinal since the classifier is also given by restriction. Let $H: \Phi(b) \to \Psi g_! (b)$ be the functor such that $f_!(b,X) = (g_!(b),H(X))$. %Note that $X$ is actually a cofibrant replacement of $\mathcal{K}_m$ in $\Phi(b)$. 
	Since $(b,X)$ is a cofibrant replacement of $\mathcal{K}_m$, $f_!(b,X) \simeq \mathbb{L}f_! (\mathcal{K}_m)$. It was proved in Theorem \ref{theoremcofinality} that the functor $F: \mathcal{X} \to \mathcal{Y}$ defined in \eqref{equationsmoothfunctor} is such that for all $y \in \mathcal{Y}$, the nerve of $y/F$ is contractible. This means that $F$ preserves homotopy colimits. So we can apply Lemma \ref{lemmareflection}, to get the weak equivalence $\mathbb{L}f_! (\mathcal{K}_m) \simeq \mathcal{K}_m$. %$\mathbb{L}f_! (\mathcal{K}_m) \simeq \mathcal{K}_m$.  $X \in \Phi(b)$ is also cofibrant. Using Theorem \ref{theoremformulaleftderived},
%	\[
%		f_!(b,X) \simeq \mathbb{L}f_! (\mathcal{K}_m) \simeq \mathcal{K}_m.% \hocolim_{\mathcal{X}} \widetilde{\mathcal{K}_m} = \hocolim_{\mathcal{Y}} \widetilde{\mathcal{K}_m}
%	\]
	We deduce that $f_!(b,X) \simeq \mathcal{K}_m$, and in particular $H(X) \simeq \mathcal{K}_m$. %Note that $X$ is actually a cofibrant replacement of $\mathcal{K}_m$ in $\Phi(b)$, so we have $H(X) \simeq X$. 
	The conclusion follows from an adjunction argument similar to \cite[Lemma 4.9]{deleger}.
%	\[
%		f_!(b,X) = \mathbb{L}f_!(b,X) = \hocolim_{\mathcal{X}} \widetilde{(b,X)}
%	\]
%	We want to apply \cite[Theorem 3.3]{bwdgrothendieck}.
\end{proof}

%\begin{lemma}
%	There is a Quillen adjunction
%	
%\end{lemma}
%
%\begin{proof}
%	We will use 
%\end{proof}

\subsection{Fibration sequence}% and rectification}

\begin{definition}
	Let $\mathcal{K}_m^{\specialdots}$ be the polynomial $2$-monad given by the polynomial
	\[
	\xymatrix{
		\mathrm{ob}(\mathfrak{D}^m) & \mathtt{UCG}^* \ar[l] \ar[r] & \mathtt{UCG} \ar[r] & \mathrm{ob}(\mathfrak{D}^m)
	}
	\]
	%	\[
	%	\xymatrix{
		%		\mathrm{ob}(\mathfrak{D}^m) & (\mathbb{K}_m^{\fivedots})^* \ar[l] \ar[r] & \mathbb{K}_m^{\fivedots} \ar[r] & \mathrm{ob}(\mathfrak{D}^m)
		%	}
	%	\]
	where $\mathtt{UCG}$ is the full subcategory of objects $(\mu,c_1,\ldots,c_k,c) \in \mathtt{PCG}$ such that if $c = t$, then there is exactly one $i \in \{1,\ldots,k\}$ such that $c_i = t$. The rest of the description is as in Definition \ref{definitionkmrotboxpolymon}.% of objects $(\mu,c_1,\ldots,c_k,c) \in \mathtt{CCG}(k)$ such that for all $ij \in {k \choose 2}$, $c_i \leq (-1,\mu_{ij})$ or $c_j \leq (1,\mu_{ij})$.
%	Let $\mathcal{K}_m^{\specialdots}$ be the sub-polynomial $2$-monads of $\mathcal{K}_m^\fivedots$ whose operations with target $(0,m+1)$ have exactly one vertex labelled with $(0,m+1)$.
\end{definition}

\begin{definition}
	Let $\mathbb{B}$ be as in Definition \ref{definitionphi} and
	\[
		\Phi^\circ: \mathbb{B} \to \mathrm{CAT}
	\]
	be the contravariant functor which sends $b \in \mathbb{B}$ to the category of objects $X \in \mathcal{M}$ together with the structure of a $\mathcal{K}_m^{\specialdots}$-algebra on the pair $(b,X)$.
\end{definition}

\begin{lemma}\label{lemmarectification}
	Any weak equivalence $b_1 \to b_2$ in $\mathbb{B}$ induces a Quillen equivalence between $\Phi^\circ(b_1)$ and $\Phi^\circ(b_2)$.
	%	\[
	%	\xymatrix{
		%		\Phi^\circ(b_1) \ar@/^/[r] \ar@{}[r]|\perp & \Phi^\circ(b_2). \ar@/^/[l]
		%	}
	%	\]
\end{lemma}

\begin{proof}
	For any $b \in \mathbb{B}$, $\Phi^\circ(b)$ is a category of presheaves, so we can apply for example \cite[Theorem 3.17 (1)]{bwdgrothendieck}.
\end{proof}

Let $I$ be the unit of $\mathcal{M}$, which is given by the collection $(1,0,0,\ldots)$, since $\mathcal{M}$ is the category of right $\mathcal{K}_m$-modules.

\begin{lemma}\label{lemmapointed}
	There is a Quillen equivalence between the category $\Phi(b)$ and the category of elements $X \in \Phi^\circ(\mathcal{K}_m)$ together with a morphism $I \to X$.% pointed equipped with a morphism $I \to X_0$.
%	The category of $\Phi(\mathcal{K}_m)$ is equivalent to the category of elements $X \in \Phi^\circ(\mathcal{K}_m)$ which are \emph{pointed}, that is equipped with a morphism from the unit of $\mathcal{M}$ to $X$.% pointed equipped with a morphism $I \to X_0$.
\end{lemma}

\begin{proof}
	We follow the same arguments as in the proof of \cite[Lemma 4.17]{deleger}. 
	First note that there is an obvious inclusion map of polynomial $2$-monads $\mathcal{K}_m^{\specialdots} \to \mathcal{K}_m^{\fivedots}$ which induces a forgetful functor $U: \Phi(b) \to \Phi^\circ(b)$. %Similarly to the situation in the proof of \cite[Lemma 3.20]{delegergrego}, l
	Let $\alpha \in \Phi^\circ(b)$ be the image of the initial object in $\Phi(b)$ through $U$. As in \cite[Lemma 3.20]{delegergrego}, the category $\Phi(b)$ is isomorphic to $\alpha/\Phi^\circ(b)$. Let $V: \Phi^\circ(b) \to \mathcal{M}$ be the forgetful functor and $\gamma \in \Phi^\circ(b)$ the image of $I$ through the left adjoint of $V$. We want to prove that there is a weak equivalence $\gamma \to \alpha$.
	
	Fortunately both $\alpha$ and $\gamma$ can be computed using classifiers. Firstly, $\alpha$ can be computed using the classifier associated to the inclusion morphism of polynomial $2$-monads $\iota: \mathcal{K}_m^{\diamonddots} \to \mathcal{K}_m^{\fivedots}$. Indeed, $\int \Phi$ is the category of algebras over $\mathcal{K}_m^{\fivedots}$ and the initial object in $\Phi(b)$ is given by $\iota_!(b)$, since $\iota_!$ is the left adjoint of the restriction functor $\iota^*: \int \Phi \to \mathbb{B}$ given by projection to the base. Secondly, $\gamma$ can be computed using the classifier associated to the inclusion morphism of polynomial $2$-monads $\kappa: \mathcal{K}_m^{\plusdots} \to \mathcal{K}_m^{\specialdots}$. Here, $\mathcal{K}_m^{\plusdots}$ is the polynomial $2$-monad given by the polynomial
	\[
		\xymatrix{
			\mathrm{ob}(\mathfrak{D}^m) & \mathtt{SCG}^* \ar[l] \ar[r] & \mathtt{SCG} \ar[r] & \mathrm{ob}(\mathfrak{D}^m)
		}
	\]
	where $\mathtt{SCG}$ is the full subcategory of objects $(\mu,c_1,\ldots,c_k,c) \in \mathtt{CCG}$ such that if $c=t$, then $k=1$, $\mu$ is the unique object of $\mathtt{CG}(1)$ and $c_1=t$. This time, $\int \Phi^\circ$ is the category of algebras over $\mathcal{K}_m^{\specialdots}$ and $\kappa^*(b,X)=(b,V(X))$, so $\kappa_!(b,I) = (b,\gamma)$.
	
	To simplify the notations, let $\mathcal{X}$ and $\mathcal{Y}$ be the classifiers associated to $\kappa$ and $\iota$ respectively. Both $\mathcal{X}$ and $\mathcal{Y}$ are actually full subcategories of the classifier associated to the morphism of polynomial $2$-monads \eqref{equationmappolymon} which was described in Lemma \ref{lemmacomputationclassifier}. $\mathcal{X}$ is the full subcategory of objects $(\mu,c_1,\ldots,c_k,c) \in \mathtt{PCG}$ such that if $c = t$, then there is exactly one $i \in \{1,\ldots,k\}$ such that $c_i = t$. $\mathcal{Y}$ is the full subcategory of objects $(\mu,c_1,\ldots,c_k,c) \in \mathtt{PCG}$ such that for all $i \in \{1,\ldots,k\}$, $c_i \neq t$. There is a functor $F: \mathcal{X} \to \mathcal{Y}$ which removes the unique $i$ such that $c_i=t$. In other words, it sends $\mu \in \mathtt{CG}(k)$ to its restriction to $\{1,\ldots,k\} \backslash \{i\}$. %$F$ has a left adjoint $G$ defined as follows. 
	
	Let us sketch a proof that $y/F$ has a contractible nerve. For $y = (\nu,d_1,\ldots,d_l,t) \in \mathcal{Y}$, an object in $y/F$ is given by $(\mu,c_1,\ldots,c_{l+1},t) \in \mathcal{X}$ together with $i \in \{1,\ldots,l+1\}$ and a map $\nu \to \mu'$, where $\mu'$ is the restriction of $\mu$ to $\{1,\ldots,l+1\} \backslash \{i\}$. The morphisms are just morphisms in $\mathcal{X}$. For example, if $m=2$, and let $A,B,C,D,E$ corresponding to the objects of $\mathfrak{D}^2$ as in Remark \ref{remarkoct}. Let $y=(\nu,B,A,E)$, with $\nu_{12}=2$. Let us draw $(\mu,c_1,c_2,c_3,E) \in \mathcal{X}$ as follows:
	
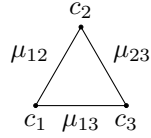
\[
		\begin{tikzpicture}[scale=.6]
		\draw[fill] (0,1.732) circle (1.2pt) node[above]{$c_2$};
		\draw[fill] (-1,0) circle (1.2pt) node[below]{$c_1$};
		\draw[fill] (1,0) circle (1.2pt) node[below]{$c_3$};
		\draw (0,0) node[below]{$\mu_{13}$};
		\draw (-.5,.76) node[above left]{$\mu_{12}$};
		\draw (.5,.76) node[above right]{$\mu_{23}$};
		\draw (-1,0) -- (1,0) -- (0,1.732) -- (-1,0);
		\end{tikzpicture}
	\]
	Then $y/F$ is the following category:%$\dagger$
	
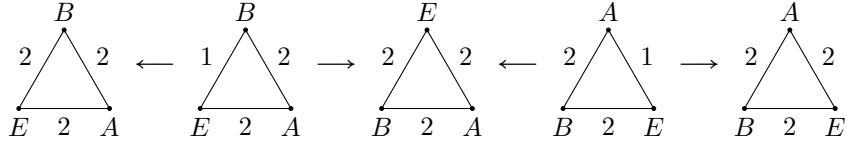
\[
	\begin{tikzpicture}[scale=.6]
	\begin{scope}[shift={(-8,0)}]
	\draw[fill] (0,1.732) circle (1.2pt) node[above]{$B$};
	\draw[fill] (-1,0) circle (1.2pt) node[below]{$E$};
	\draw[fill] (1,0) circle (1.2pt) node[below]{$A$};
	\draw (0,0) node[below]{$2$};
	\draw (-.5,.76) node[above left]{$2$};
	\draw (.5,.76) node[above right]{$2$};
	\draw (-1,0) -- (1,0) -- (0,1.732) -- (-1,0);
	\end{scope}
	
	\draw (-6,.95) node{$\longleftarrow$};
	
	\begin{scope}[shift={(-4,0)}]
	\draw[fill] (0,1.732) circle (1.2pt) node[above]{$B$};
	\draw[fill] (-1,0) circle (1.2pt) node[below]{$E$};
	\draw[fill] (1,0) circle (1.2pt) node[below]{$A$};
	\draw (0,0) node[below]{$2$};
	\draw (-.5,.76) node[above left]{$1$};
	\draw (.5,.76) node[above right]{$2$};
	\draw (-1,0) -- (1,0) -- (0,1.732) -- (-1,0);
	\end{scope}
	
	\draw (-2,.95) node{$\longrightarrow$};
	
	\draw[fill] (0,1.732) circle (1.2pt) node[above]{$E$};
	\draw[fill] (-1,0) circle (1.2pt) node[below]{$B$};
	\draw[fill] (1,0) circle (1.2pt) node[below]{$A$};
	\draw (0,0) node[below]{$2$};
	\draw (-.5,.76) node[above left]{$2$};
	\draw (.5,.76) node[above right]{$2$};
	\draw (-1,0) -- (1,0) -- (0,1.732) -- (-1,0);
	
	\draw (2,.95) node{$\longleftarrow$};
	
	\begin{scope}[shift={(4,0)}]
	\draw[fill] (0,1.732) circle (1.2pt) node[above]{$A$};
	\draw[fill] (-1,0) circle (1.2pt) node[below]{$B$};
	\draw[fill] (1,0) circle (1.2pt) node[below]{$E$};
	\draw (0,0) node[below]{$2$};
	\draw (-.5,.76) node[above left]{$2$};
	\draw (.5,.76) node[above right]{$1$};
	\draw (-1,0) -- (1,0) -- (0,1.732) -- (-1,0);
	\end{scope}
	
	\draw (6,.95) node{$\longrightarrow$};
	
	\begin{scope}[shift={(8,0)}]
	\draw[fill] (0,1.732) circle (1.2pt) node[above]{$A$};
	\draw[fill] (-1,0) circle (1.2pt) node[below]{$B$};
	\draw[fill] (1,0) circle (1.2pt) node[below]{$E$};
	\draw (0,0) node[below]{$2$};
	\draw (-.5,.76) node[above left]{$2$};
	\draw (.5,.76) node[above right]{$2$};
	\draw (-1,0) -- (1,0) -- (0,1.732) -- (-1,0);
	\end{scope}
	\end{tikzpicture}
	\]
	For any $y \in \mathcal{Y}$, $y/F$ is non-empty and any two elements in this category are connected by a zigzag as above. There are no loops because between any two elements there is always a shortest possible zigzag.

	According to Theorem \ref{theoremformulaleftderived}, $\gamma$ and $\alpha$ can be computed as homotopy colimits over the classifiers $\mathcal{X}$ and $\mathcal{Y}$ respectively. The map $\xi: \gamma \to \alpha$ is induced by $F: \mathcal{X} \to \mathcal{Y}$. Since $y/F$ has a contractible nerve for all $y \in \mathcal{Y}$, $F$ preserves homotopy limits, which means that $\xi$ is indeed a weak equivalence.
	
	Let us finish by gathering all the arguments together. The category $\Phi(b)$ is isomorphic to $\alpha/\Phi^\circ(b)$. Since $\Phi^\circ(b)$ is a category of presheaves, is it left proper. So, according to \cite[Proposition 2.7]{rezk}, there is a Quillen equivalence between $\alpha/\Phi^\circ(b)$ and $\gamma/\Phi^\circ(b)$. Finally, using Lemma \ref{lemmarectification}, there is a Quillen equivalence between $\gamma/\Phi^\circ(b)$ and $\gamma/\Phi^\circ(\mathcal{K}_m)$.
\end{proof}

\begin{lemma}\label{lemmafibrationsequence}
	For any %$b \in \mathbb{B}$ and any 
	map $X \to Y$ in $\Phi(b)$, there is a fibration sequence
	\[
	\Phi(b)^\mathsf{h}(X,Y) \to \Phi^\circ(\mathcal{K}_m)^\mathsf{h}(X,Y) \to Y_0.
	\]
\end{lemma}

\begin{proof}
	It is a direct consequence of Lemma \ref{lemmapointed} %applied to $b=\mathcal{K}_m$ 
	and \cite[Theorem 4.14]{deleger}.
\end{proof}

\subsection{Equivalences of categories}

% in a symmetric monoidal category $(\mathcal{V},\otimes,v)$. Recall that a \emph{left $\mathcal{A}$-module} is given by a collection $\mathcal{C} := (\mathcal{C}(k))_{k \geq 0}$ of objects in $\mathcal{V}$ together with maps

Let $\mathcal{A}$ and $\mathcal{B}$ be two (one-coloured) symmetric operads in a symmetric monoidal category $(\mathcal{V},\otimes,v)$. Let $\mathcal{C}$ and $\mathcal{D}$ be two $\mathcal{A}-\mathcal{B}$-bimodules. Recall \cite[Definition 3.2]{deleger} that an \emph{infinitesimal $\mathcal{C}-\mathcal{D}$-bimodule} $\mathcal{E}$ is given by a collection $\mathcal{E} := (\mathcal{E}(k))_{k \geq 0}$ of objects in $\mathcal{V}$ together with maps
\begin{equation}\label{equationleftaction}
	\mathcal{A}(k) \times \mathcal{C}(n_1) \times \ldots \times \mathcal{C}(n_{i-1}) \times \mathcal{E}(n_i) \times \mathcal{D}(n_{i+1}) \times \ldots \times \mathcal{D}(n_k) \to \mathcal{E}(n_1+\ldots+n_k)
\end{equation}
%\[
%	A_k \times C_{n_1} \times \ldots \times C_{n_{i-1}} \times E_{n_i} \times D_{n_{i+1}} \times \ldots \times D_{n_k} \to E_{n_1+\ldots+n_k}
%\]
and
\[
	\mathcal{E}(k) \times \mathcal{B}(n_1) \times \ldots \times \mathcal{B}(n_k) \to \mathcal{E}(n_1+\ldots+n_k)
\]
satisfying axioms. Also recall that an \emph{infinitesimal $\mathcal{A}$-bimodule} $\mathcal{E}$ is given by a collection $(\mathcal{E}(k))_{k \geq 0}$ of objects in $\mathcal{V}$ together with maps
\[
	\circ_i : \mathcal{A}(k) \otimes \mathcal{E}(n) \to \mathcal{E}(k+n-1)
\]
and
\[
	\bullet_i : \mathcal{E}(k) \otimes \mathcal{A}(n) \to \mathcal{E}(k+n-1),
\]
for $1 \leq i \leq k$, satisfying axioms.

\begin{lemma}\label{lemmaphikm}
	The category $\Phi^\circ(\mathcal{K}_m)$ is equivalent to the category of infinitesimal $\mathcal{K}_m$-bimodules.
\end{lemma}

\begin{proof}
	We will prove that the category $\Phi^\circ(\mathcal{K}_m)$ is equivalent to the category of infinitesimal $\mathcal{K}_m-\mathcal{K}_m$-bimodules in the sense of \cite[Definition 3.2]{deleger}. The conclusion follows from \cite[Lemma 3.3]{deleger}, which says that we recover the classical notion of infinitesimal $\mathcal{K}_m$-bimodules in this case.
	
	An object of $\Phi^\circ(\mathcal{K}_m)$ is given by a right $\mathcal{K}_m$-module $X$ together with the structure of a $\mathcal{K}_m^{\specialdots}$-algebra on the pair $(\mathcal{K}_m,X)$. Let us unpack the definition of a $\mathcal{K}_m^{\specialdots}$-algebra. For $(\mu,c_1,\ldots,c_k,c) \in \mathtt{UCG}$ where $c=t$, there is exactly one $i \in \{1,\ldots,k\}$ such that $c_i=t$. This induces a map
	\[
		\underbrace{\mathcal{K}_m \otimes \ldots \otimes \mathcal{K}_m}_{i-1} \otimes X \otimes \underbrace{\mathcal{K}_m \otimes \ldots \otimes \mathcal{K}_m}_{k-i} \to X
	\]
	in the category of right $\mathcal{K}_m$-modules. Since the tensor product is given by convolution, we get functors
	\begin{equation}\label{equationconvolution}
		\mathcal{K}_m(n_1) \times \ldots \times X(n_i) \times \ldots \times \mathcal{K}_m(n_k) \to X(n_1+\ldots+n_k).
	\end{equation}
	
	On the other hand, an infinitesimal $\mathcal{K}_m-\mathcal{K}_m$-bimodule is given by a right $\mathcal{K}_m$-module $X$ together with left actions as in \eqref{equationleftaction}. In our case, $\mathcal{A}(k)=\mathcal{K}_m(k)$, so we get functors as in \eqref{equationconvolution} but this time for any object $\mu \in \mathcal{K}_m(k)$. It remains to prove that the left actions are equivalent.
	
	First, let $X \in \Phi^\circ (\mathcal{K}_m)$. Let us construct a left action as in \eqref{equationleftaction}. As we know from Proposition \ref{propositionequivalencealgebras}, $\mathcal{K}_m$ can be seen both as categorical operad and a polynomial $2$-monad. So $\mu \in \mathcal{K}_m(k)$ is given by an object in $\mathtt{CG}(k)$. Let $i \in \{1,\ldots,k\}$. For $j \in \{1,\ldots,k\}$, let $c_j=(-1,\mu_{ji})$ if $j<i$, $c_j=t$ if $j=i$ and $c_j=(1,\mu_{ij})$ is $j>i$.  otherwise. This gives an object $(\mu,c_1,\ldots,c_k,t) \in \mathtt{UCG}$ and we can defined the left action map as the one induced by this object. Note that if $c'_1,\ldots,c'_k$ are such that $c'_i=t$ and $(\mu,c'_1,\ldots,c'_k,t) \in \mathtt{UCG}$, then $c'_j \leq c_j$ for all $j \in \{1,\ldots,k\}$.
	
	On the other hand, if $X$ is an infinitesimal $\mathcal{K}_m-\mathcal{K}_m$-bimodule and $(\mu,c_1,\ldots,c_k,t) \in \mathtt{UCG}$, then we can take the functor \eqref{equationleftaction} associated to $\mu$.
	
	We obtain two functors inverse of each other. The fact that they are inverse comes from the fact that $\mathcal{K}_m \in \mathbb{B}$ is actually given by the constant object $\mathcal{K}_m \in \mathcal{M}$, which means that for fixed $\mu \in \mathtt{CG}$, the left action \eqref{equationleftaction} is the same for any $(c_1,\ldots,c_k)$ such that $(\mu,c_1,\ldots,c_k,t) \in \mathtt{UCG}$.% Note that for any $i \in \ldots \{1,\ldots,k\}$, $(\mu,c_1,\ldots,c_k,t) \in \mathtt{UCG}$ and $(\mu,c'_1,\ldots,c'_k,t) \in \mathtt{UCG}$, such that $c_i=c'_i=t$, there is $(\mu,c''_1,\ldots,c''_k,t) \in \mathtt{UCG}$ such that $c_i \leq c''_i$ and $c'_i \leq c''_i$ for all $i \in \{1,\ldots,k\}$. Indeed, we can take  %$rs \in {k \choose 2}$, let 
\end{proof}

%\subsection{Rectifications}

\begin{lemma}\label{lemmapsikm}
%	For any right $\mathcal{K}_m$-module $R$, t
	The category $\Psi(D^0(\mathcal{K}_m))$ is equivalent to the category of $\mathcal{K}_m$-bimodules $X$ equipped with a map $S^{m-1}(\mathcal{K}_m) \to X$ of $\mathcal{K}_m$-bimodules.%$S^{m-1}(\mathcal{K}_m) / \mathrm{Bimod}_{\mathcal{K}_m}$.%of $\mathcal{K}_m$-bimodules $X$ equipped with a map $S^{m-1}(\mathcal{K}_m) \to X$ of $\mathcal{K}_m$-bimodules.
\end{lemma}

\begin{proof}
	The category $\mathbb{C}$ of $\mathcal{K}_m^{\rotbox}$-algebras in the category of covariant presheaves over $\mathfrak{S}^{m-1}$ in the category of $\mathcal{K}_m$-algebras, by definition of the Boardman-Vogt tensor product. For $c \in \mathbb{C}$ and $Y \in M$, the structure of a $\mathcal{K}_m^{\rotboxtimes}$-algebra on the pair $(c,Y)$ corresponds to a map $\colim_{\mathfrak{S}^{m-1}} c \to Y$. So $\Psi(c)$ is equivalent to $\colim_{\mathfrak{S}^{m-1}} c / \mathrm{Bimod}_{\mathcal{K}_m}$. We have
	\[
		\colim_{\mathfrak{S}^{m-1}} D^0(\mathcal{K}_m) / \mathrm{Bimod}_{\mathcal{K}_m} \simeq \hocolim_{\mathfrak{S}^{m-1}} \mathcal{K}_m / \mathrm{Bimod}_{\mathcal{K}_m} \simeq S^{m-1}(\mathcal{K}_m) / \mathrm{Bimod}_{\mathcal{K}_m}
	\]
	where the second equivalence is given by Lemma \ref{lemmaspherehocolimformula}. %, $\hocolim_{\mathfrak{S}^{m-1}} \mathcal{K}_m \simeq S^{m-1}(\mathcal{K}_m)$, which 
	This concludes the proof.
\end{proof}

%Test $\mathclap{\scalebox{0.5}{$\circ$}}\diamonddots$
%$\mathrel{\ooalign{\hss$\diamonddots$ \hss\cr\adjustbox{lap={\width}{.175em},raise=0.13em}{\scalebox{0.5}{$\circ$}}\cr}}$
%\stackon[-3.5pt]{$\diamonddots$}{${\scalebox{0.5}{$\circ$}}$}
%$\equaldot$
%%\newcommand{\eq}{\mathrel{{=}\mkern-9.5mu{\cdot}\mkern4mu}}
%$\mathrel{{=}\mkern-9.5mu{\cdot}\mkern4mu}$
%
%We have the following triangle of adjunctions
%\[
%	\xymatrix{
%		\Phi^\circ(b) \ar@/_/[rr]_U \ar@/_/[rd] && \Phi(b) \ar@/^/[ld] \\
%		& \mathrm{Cat} \ar@/^/[ru]^G
%	}
%\]
%
%\begin{lemma}
%	There is an isomorphism $U(0) \sim G(1)$ in $\Phi(b)$.
%\end{lemma}
%
%\begin{proof}
%	$U(0)$ and $G(1)$ can both be computed as classifiers. $U(0)$ is the category of properly labelled complete graphs but without vertex $E$. $G(1)$ is the category of properly labelled complete graphs with exactly one vertex $E$. There is a forgetful functor $G(1) \to U(0)$ which deletes the vertex $E$. 
%\end{proof}
%
%\subsection{Quasi-tameness}
%
%\begin{lemma}
%	The polynomial $2$-monad $\mathcal{K}_m^\fivedots$ is quasi-tame.
%\end{lemma}
%
%\begin{proof}
%	We need to compute the classifier.
%\end{proof}
%
%Let $\mathcal{C}_m^\fivedots$
%
%
%
%\begin{theorem}
%	There is a weak equivalence
%	\[
%		\Omega \mathrm{Sop} (\mathcal{C}_m,\mathcal{C}_n) \to \mathrm{Sop}_{\mathcal{C}_m \sqcup \mathcal{C}_m} (\mathcal{C}_m,\mathcal{C}_n) \to \mathrm{Bimod}_{\mathcal{C}_m,\mathcal{C}_m} (\mathcal{C}_m,\mathcal{C}_n)
%	\]
%\end{theorem}

\subsection{Putting everything together}

Let $\mathrm{IBimod}_{\mathcal{K}_m}$ be the category of infinitesimal $\mathcal{K}_m$-bimodules.

\begin{theorem}\cite[Main Theorem 1]{ducoulombierturchin}
	Let $Y$ be a $\mathcal{K}_m$-bimodule equipped with a map of $\mathcal{K}_m$-bimodules $\mathcal{K}_m \to Y$. If $Y_0$ is contractible, then there is a weak equivalence
	\[
		\Omega^m \mathrm{Bimod}_{\mathcal{K}_m}^\mathsf{h}(\mathcal{K}_m,Y) \xrightarrow{\sim} \mathrm{IBimod}_{\mathcal{K}_m}^\mathsf{h}(\mathcal{K}_m,Y).% \to X_0.
	\]
\end{theorem}

\begin{proof}
%	Applying Lemma \ref{lemmageneraldelooping} and Lemma \ref{lemmapsikm}, we get the delooping
%	\[
%		\Omega^m \mathrm{Bimod}_{\mathcal{K}_m}(\mathcal{K}_m,X) \to \nsfrac{S^{m-1}(\mathcal{K}_m)}{\mathrm{Bimod}_{\mathcal{K}_m}} (\mathcal{K}_m,X) = \Psi(D^0(\mathcal{K}_m)) (\mathcal{K}_m,X).
%	\]
%	The category of $\mathcal{K}_m$-bimodules is left proper according to Lemma \ref{lemmaleftproper}. 
	We have the following sequence of weak equivalences
	\begin{align*}
		\Omega^m \mathrm{Bimod}_{\mathcal{K}_m}^\mathsf{h}(\mathcal{K}_m,Y) &\xrightarrow{\sim} (S^{m-1}(\mathcal{K}_m) / \mathrm{Bimod}_{\mathcal{K}_m})^\mathsf{h} (\mathcal{K}_m,Y) && \text{Lemma \ref{lemmageneraldelooping}}\\
		&\xrightarrow{\sim} \Psi(b)^\mathsf{h} (\mathcal{K}_m,Y) && \text{Lemma \ref{lemmapsikm}}\\
		&\xrightarrow{\sim} \Phi(b)^\mathsf{h} (\mathcal{K}_m,Y) && \text{Lemma \ref{lemmaquillenreflection}}\\
		&\xrightarrow{\sim} \Phi^\circ(b)^\mathsf{h} (\mathcal{K}_m,Y) && \text{Lemma \ref{lemmafibrationsequence}}\\
		&\xrightarrow{\sim} \Phi^\circ(\mathcal{K}_m)^\mathsf{h} (\mathcal{K}_m,Y) && \text{Lemma \ref{lemmarectification}}\\
		&\xrightarrow{\sim} \mathrm{IBimod}_{\mathcal{K}_m}^\mathsf{h} (\mathcal{K}_m,Y) && \text{Lemma \ref{lemmaphikm}}
	\end{align*}
\end{proof}

By taking $Y=\mathcal{K}_n$ in the previous theeorem, we get:

\begin{corollary}
%	For $n \geq m \geq 0$, t
	There is a weak equivalence
	\[
		\Omega^m \mathrm{Bimod}_{\mathcal{K}_m}^\mathsf{h}(\mathcal{K}_m,\mathcal{K}_n) \to \mathrm{IBimod}_{\mathcal{K}_m}^\mathsf{h}(\mathcal{K}_m,\mathcal{K}_n).
	\]
\end{corollary}

\bibliographystyle{plain}
\bibliography{homotopy-theory}

\end{document}